\long\def\symbolfootnote[#1]#2{\begingroup\def\thefootnote{\fnsymbol{footnote}}\footnote[#1]{#2}\endgroup}

\documentclass{amsart}
\usepackage{amssymb}
\usepackage{amsmath}
\usepackage{amsfonts}

\newtheorem{theorem}{Theorem}[section]

\newtheorem{lemma}[theorem]{Lemma}
\theoremstyle{remark}

\theoremstyle{definition}

\newtheorem{claim}[theorem]{Claim}

\numberwithin{equation}{section}

\begin{document}
\author{Ovidiu Munteanu and Jiaping Wang}
\title[Weighted Laplacian and applications]{Analysis of weighted Laplacian
and applications to Ricci solitons}
\date{}

\begin{abstract}
We study both function theoretic and spectral properties of the weighted
Laplacian $\Delta_f$ on complete smooth metric measure space $(M,g,e^{-f}dv)$
with its Bakry-\'{E}mery curvature $Ric_f$ bounded from below by a constant.
In particular, we establish a gradient estimate for positive $f-$harmonic
functions and a sharp upper bound of the bottom spectrum of $\Delta_f$ in
terms of the lower bound of $Ric_{f}$ and the linear growth rate of $f.$ We
also address the rigidity issue when the bottom spectrum achieves its
optimal upper bound under a slightly stronger assumption that the gradient
of $f$ is bounded.

Applications to the study of the geometry and topology of gradient Ricci
solitons are also considered. Among other things, it is shown that the
volume of a noncompact shrinking Ricci soliton must be of at least linear
growth. It is also shown that a nontrivial expanding Ricci soliton must be
connected at infinity provided its scalar curvature satisfies a suitable
lower bound.
\end{abstract}

\maketitle

\section{Introduction}

\symbolfootnote[0]{The first author has been partially supported by NSF grant No. DMS-1005484}

In our previous paper \cite{MW}, we have studied some function theoretic and
spectral properties of the weighted Laplacian on a smooth metric measure
space with nonnegative Bakry-\'{E}mery curvature. We have also applied the
results to conclude a nontrivial steady gradient Ricci soliton must be
connected at infinity. The purpose of this sequel to \cite{MW} is twofold.
The first is to continue our study of the weighted Laplacian on a smooth
metric measure space, now under the more general assumption that its Bakry-%
\'{E}mery curvature is bounded from below by a negative constant. The second
is to demonstrate that the results and techniques from such study lead to
geometric and topological information of shrinking and expanding gradient
Ricci solitons.

Recall that a smooth metric measure space, denoted by $\left(M,g,e^{-f}dv%
\right)$ throughout the paper, is a Riemannian manifold $\left( M,g\right) $
together with a weighted volume form $e^{-f}dv,$ where $f$ is a smooth
function on $M$ and $dv$ the volume element induced by the Riemannian metric 
$g.$ The associated weighted Laplacian $\Delta _f$ is given by

\begin{equation*}
\Delta _{f}u:=\Delta u-\left\langle \nabla f,\nabla u\right\rangle ,
\end{equation*}%
which is a self-adjoint operator on the space of square integrable functions
on $M$ with respect to the measure $e^{-f}dv.$ A function $u$ is called $f-$%
harmonic if $\Delta _{f}u=0.$ It is easy to see that $f-$harmonic functions
are characterized as the critical points of the weighted Dirichlet energy $%
\int_{M}|\nabla u|^{2}e^{-f}dv.$

The Bakry-\'{E}mery curvature $Ric_f$ associated to smooth metric measure
space $\left(M,g,e^{-f}dv\right)$ is defined \cite{BE} by

\begin{equation*}
Ric_{f}:=Ric+Hess\left( f\right) ,
\end{equation*}%
where $Ric$ denotes the Ricci curvature of $(M,g)$ and $Hess\left( f\right) $
the Hessian of $f.$

The weighted Laplacian and the Bakry-\'{E}mery curvature are natural objects
in the geometric analysis. The most significant and interesting case of our
concern here is the so-called gradient Ricci solitons. Recall a complete
manifold $(M,g)$ is a gradient Ricci soliton if the equation $%
Ric_{f}=\lambda g$ holds for some function $f$ and scalar $\lambda.$ The
soliton is called expanding, steady and shrinking, accordingly, if $%
\lambda<0,$ $\lambda=0$ and $\lambda>0.$ It is customary to normalize the
constant $\lambda \in \left\{ -1/2,0,1/2\right\} $ by scaling the metric $g.$
As suggested by the name, the gradient Ricci solitons arise from the study
of Ricci flows, particularly from the blow up analysis of the singularities
of the Ricci flows \cite{H}. It is thus a central issue in the study of
Ricci flows to understand and classify gradient Ricci solitons. Note that
the Ricci soliton equation $Ric_{f}=\lambda g$ reduces to the Einstein
equation $Ric=\lambda g$ when $f$ is a constant function. So the soliton
equation is also of own interest as a geometric partial differential
equation. We refer the readers to the book \cite{CLN} for more information
on gradient Ricci solitons.

The Bakry-\'{E}mery curvature is closely related to the weighted Laplacian
as indicated by the following Bochner type identity.

\begin{equation*}
\Delta_f |\nabla u|^2=2|Hess(u)|^2+2\left\langle \nabla u,\nabla \Delta_f
u\right\rangle+2Ric_{f} (\nabla u,\ \nabla u).
\end{equation*}%
This is of course very much in parallel to how the Ricci curvature is
related to the Laplacian on a complete manifold. Taking this point of view,
in the first part of the paper, we will develop some analogous results to
the Laplacian for the weighted Laplacian under the assumption that the Bakry-%
\'{E}mery curvature is bounded from below. However, we would like to point
out that unlike the classical case on the analysis of the Laplacian, the
function $f$ enters into play in the behavior of the weighted Laplacian $%
\Delta_f.$ The assumptions on $f$ dictate both the conclusions and the level
of technical difficulties involved in the arguments. As in \cite{MW}, we
continue to assume that $f$ grows at most linearly, that is,

\begin{equation*}
\left\vert f\right\vert \left( x\right) \leq \alpha r\left( x\right) +\beta
\end{equation*}%
for some constants $\alpha \ge 0$ and $\beta \ge 0,$ where $r\left( x\right)
:=d\left(p,x\right) $ is the geodesic distance to a fixed point $p$ in $M.$
The linear growth rate $a$ of $f$ is then defined to be the infimum over all
such $\alpha.$

Also, as indicated, we assume that $Ric_{f}$ is bounded below by a negative
constant. After a suitable scaling of the metric, we may in fact assume $%
Ric_{f}\geq -\left( n-1\right) ,$ where $n$ is the dimension of $M.$

Our first result is a gradient estimate for positive $f-$harmonic functions
on $\left( M,g,e^{-f}dv\right).$

\begin{theorem}
\label{Gradient_Est} Let $\left( M,g,e^{-f}dv\right) $ be a smooth metric
measure space of dimension $n$ with $Ric_{f}\geq -\left( n-1\right) .$
Assume that there exists constant $a>0$ such that the oscillation of $f$
over the unit ball $B_x(1)$ for any $x\in M$ satisfies 
\begin{equation*}
\sup_{y\in B_{x}\left( 1\right) }\left\vert f\left( y\right) -f\left(
x\right) \right\vert \leq a.
\end{equation*}%
Then there exists a constant $C\left( n,a\right) $ depending only on $n$ and 
$a$ such that for any $u>0$ with $\Delta _{f}u=0$ we have%
\begin{equation*}
\left\vert \nabla \log u\right\vert \leq C\left( n,a\right) .
\end{equation*}
\end{theorem}

We remark that the assumption on $f$ in Theorem \ref{Gradient_Est} is
satisfied for example if $\left\vert \nabla f\right\vert \leq a$ or if $f$
is bounded on $M.$ In the case that $\left\vert \nabla f\right\vert \leq a,$
the result can be proved more or less following the classical argument of
Yau \cite{Y} via the aforementioned Bochner identity. The details have been
carried out in \cite{Wu}. If $Ric_{f}\geq 0,$ then the result has been
proved in \cite{MW}. In fact, the following stronger conclusion holds. 
\begin{equation*}
\left\vert \nabla \log u\right\vert \leq C(n)\,a.
\end{equation*}

Our second result concerns the bottom spectrum of the weighted Laplacian $%
\Delta_f.$ Let $\lambda _{1}\left( \Delta _{f}\right) :=\inf Spec\left(
-\Delta _{f}\right) .$ Then the variational characterization for $\lambda
_{1}\left( \Delta _{f}\right) $ implies that 
\begin{equation*}
\lambda _{1}\left( \Delta _{f}\right) =\inf_{\phi \in C_{0}^{\infty }\left(
M\right) }\frac{\int_{M}\left\vert \nabla \phi \right\vert ^{2}e^{-f}\,dv}{%
\int_{M}\phi ^{2}e^{-f}\,dv}.
\end{equation*}

In the case $Ric_{f}\geq 0,$ it was shown in \cite{MW} that $\lambda
_{1}\left( \Delta _{f}\right) $ has an optimal upper bound of the form $%
\lambda _{1}\left( \Delta _{f}\right) \leq \frac{1}{4}a^{2},$ where $a$ is
the linear growth rate of $f.$ Here is a more general result.

\begin{theorem}
\label{Estimate} Let $\left( M^{n},g,e^{-f}dv\right) $ be a complete smooth
metric measure space with $Ric_{f}\geq -\left( n-1\right).$ Then we have 
\begin{equation*}
\lambda _{1}\left( \Delta _{f}\right) \leq \frac{1}{4}\left( n-1+a\right)
^{2},
\end{equation*}%
where $a\geq 0$ is the linear growth rate of $f.$ In particular, if $f$ is
of sublinear growth, then the bottom spectrum of $\Delta_f$ satisfies the
following sharp upper bound%
\begin{equation*}
\lambda _{1}\left( \Delta _{f}\right) \leq \frac{\left( n-1\right) ^{2}}{4}.
\end{equation*}
\end{theorem}

This result is consistent with Cheng's well-known estimate \cite{C} in the
case $f$ is constant, which says that $\lambda _{1}\left( M\right) $ is
bounded above by $\frac{\left( n-1\right) ^{2}}{4}.$ This incidentally
indicates our estimate is sharp. We remark that under the stronger
assumption that $\left\vert \nabla f\right\vert \leq a,$ the result has also
been established in \cite{SZ}.

Motivated by the work of P. Li and the second author in \cite{LW0, LW, LW1,
LW2} and our generalization in \cite{MW} to the case of weighted Laplacian
with $Ric_{f}\geq 0,$ we study the structure of manifolds on which $\lambda
_{1}\left( \Delta _{f}\right) $ achieves its maximal value in the preceding
estimate. Here, we need to impose a stronger assumption on $f$ that its
gradient is bounded.

\begin{theorem}
\label{Rigidity1} Let $\left( M,g,e^{-f}dv\right) $ be a complete smooth
metric measure space of dimension $n\geq 3$ with $Ric_{f}\geq -\left(
n-1\right) .$ Assume that $\left\vert \nabla f\right\vert \leq a$ on $M$ for
some constant $a\geq 0.$ If $\lambda _{1}\left( \Delta _{f}\right) =\frac{1}{%
4}\left( n-1+a\right) ^{2},$ then either $M$ is connected at infinity or $f$
is constant and $M$ splits as a warped product $M=\mathbb{R}\times N$ with $%
ds_{M}^{2}=dt^{2}+h^{2}\left( t\right) ds_{N}^{2},$ where $N$ is compact and
the function $h\left( t\right) =e^{t}$ if $n\geq 4$ and $h\left( t\right)
=e^{t}$ or $h\left( t\right) =\cosh t$ if $n=3.$
\end{theorem}

Let us point out that in the case $M$ is the warped product, its bottom
spectrum has maximal value $\frac{(n-1)^2}{4}.$ Let us also point out that
this result has been independently proved by Su and Zhang in \cite{SZ}.

In the second part of this paper, we consider some applications of our study
of the weighted Laplacian to gradient Ricci solitons. We first address the
issue whether a non-trivial expanding gradient Ricci soliton must be
connected at infinity. Recall that an expanding gradient Ricci soliton is a
manifold $\left( M,g\right) $ such that $Ric_{f}=-\frac{1}{2}g$ for some
function $f.$ It is known \cite{PRS} that the scalar curvature $S\geq -\frac{%
n}{2}$ on an expanding gradient Ricci soliton.

We have the following result concerning this issue.

\begin{theorem}
\label{Rigidity_Exp}Let $\left( M,g,f\right) $ be an expanding gradient
Ricci soliton. Assume that $S\geq -\frac{n-1}{2}$ on $M$. Then either $M$ is
connected at infinity or $M$ is isometric to the product $\mathbb{R}\times
N, $ where $N$ is a compact Einstein manifold and $\mathbb{R}$ the Gaussian
expanding Ricci soliton.
\end{theorem}

Note that in the second case of $M$ being a cylinder, its scalar curvature $%
S=-\frac{n-1}{2}.$ This somewhat explains why we impose such an assumption
on $S.$ However, at this point it is unclear to us whether the assumption is
in fact superfluous for nontrivial expanding gradient Ricci solitons, though
obviously there are Einstein manifolds with infinitely many ends.

As for the proof of Theorem \ref{Rigidity_Exp}, it does not follow directly
from our preceding rigidity theorem. This is because for expanding gradient
Ricci solitons, the potential function $f$ is never of linear growth unless
it is trivial according to the following result.

\begin{theorem}
\label{Growth}Let $\left( M,g,f\right) $ be a nontrivial expanding gradient
Ricci soliton. Then for all $r>2,$ 
\begin{equation*}
\frac{1}{4}\,r^{2}-Cr^{\frac{3}{2}}\,\sqrt{\ln r}\leq \sup_{\partial
B_{p}\left( r\right) }(-f)\left( x\right) \leq \frac{1}{4}r^{2}+Cr
\end{equation*}%
for some constant $C.$
\end{theorem}

At this point, it seems interesting to compare our result with the case of
shrinking gradient Ricci solitons. It has been shown by Cao and Zhou \cite%
{CZ} that

\begin{equation*}
\frac{1}{4} \left(r(x) -c\right)^2\leq f(x) \leq \frac{1}{4} \left(r(x)
+c\right)^2
\end{equation*}%
on a nontrivial, noncompact, shrinking gradient Ricci soliton. We also point
out that it was first observed in \cite{PRS} that the gradient of the
potential function $f$ must be unbounded for a nontrivial expanding gradient
Ricci soliton.

Another issue we resolve here is about the volume growth lower bound for
shrinking gradient Ricci solitons. As well known, volume growth rate is an
important piece of geometric information. In \cite{CZ}, it was proved that
the volume of a shrinking gradient Ricci soliton is at most of polynomial
growth of order $n,$ the dimension of the underlying manifold. Concerning
the lower bound, when the Ricci curvature is bounded, then it is known \cite%
{CN} the volume grows at least linearly. The question whether this is the
case for general shrinking gradient Ricci solitons has been posed to the
authors by Huaidong Cao and Lei Ni, respectively. We confirm this to be the
case here. Note that this is sharp as shown by the cylinder examples.
Indeed, for $M=\mathbb{R}\times N^{n-1},$ where $N^{n-1}$ is an Einstein
manifold such that $Ric_{N}=\frac{1}{2}g_{N}$ and $\mathbb{R}$ is the
Gaussian shrinking soliton with potential function $f=\frac{1}{4}\left\vert
x\right\vert ^{2},$ the volume of $M$ grows linearly.

\begin{theorem}
\label{Volume}Let $\left( M,g,f\right) $ be a noncompact shrinking gradient
Ricci soliton. Then there exists a constant $C>0$ such that 
\begin{equation*}
V\left( B_{p}\left( r\right) \right) \geq Cr\ \ \ \text{for all }r>0.
\end{equation*}
\end{theorem}

The paper is organized as follows. In Section \ref{Vol_Comp}, we prove
Theorem \ref{Estimate} after some discussions on Laplacian and volume
comparison results. We then prove Theorem \ref{Gradient_Est} in Section \ref%
{f_harmonic}. In Section \ref{Rigid}, we study the structure of manifolds
with maximal bottom spectrum $\lambda _{1}(\Delta _{f})$ and prove Theorem %
\ref{Rigidity1}. In Section \ref{Expanding}, we consider the expanding Ricci
solitons and prove both Theorem \ref{Rigidity_Exp} and Theorem \ref{Growth}.
In the final Section \ref{Shrinkers}, we deal with the shrinking Ricci
solitons and prove Theorem \ref{Volume}.

We would like to thank Huai-Dong Cao for his interest and stimulating
comments, which lead us to improve both Theorem 5.1 and Theorem 6.1.

\section{Volume comparison theorem\label{Vol_Comp}}

In this section, following \cite{W}, we discuss Laplacian and volume
comparison results by assuming a lower bound on the Bakry-\'{E}mery
curvature tensor. As an immediate consequence, we obtain an upper bound
estimate for the bottom spectrum of $\Delta _{f}.$

Let $\left( M,g,e^{-f}dv\right) $ be a smooth metric measure space. Take any
point $x\in M$ and express the volume form in the geodesic polar coordinates
centered at $x$ as 
\begin{equation*}
dV|_{\exp _{x}\left( r\xi \right) }=J\left( x,r,\xi \right) drd\xi
\end{equation*}%
for $r>0$ and $\xi \in S_{x}M,$ a unit tangent vector at $x.$ It is well
known that if $y\in M$ is another point such that $y=\exp _{x}$ $\left( r\xi
\right) ,$ then 
\begin{equation*}
\Delta d\left( x,y\right) =\frac{J^{\prime }\left( x,r,\xi \right) }{J\left(
x,r,\xi \right) }\text{ \ and \ }\Delta _{f}d\left( x,y\right) =\frac{%
J_{f}^{\prime }\left( x,r,\xi \right) }{J_{f}\left( x,r,\xi \right) },
\end{equation*}%
where $J_{f}\left( x,r,\xi \right) :=e^{-f}J\left( x,r,\xi \right) $ is the $%
f$-volume form in the geodesic polar coordinates. For a set $\Omega $ we
will denote by $V\left( \Omega \right) $ the volume of $\Omega $ with
respect to the usual volume form $dv,$ and $V_{f}\left( \Omega \right) $ the 
$f$-volume of $\Omega .$

\begin{lemma}
\label{Vol_Ric_Neg} Let $\left( M,g,e^{-f}dv\right) $ be a complete smooth
metric measure space with $Ric_{f}\geq -\left( n-1\right).$ Assume for some
nonnegative constants $\alpha$ and $\beta,$ 
\begin{equation*}
\left\vert f\right\vert \left( x\right) \leq \alpha r\left( x\right) +\beta
\end{equation*}%
for $x\in M.$ Then there exists a constant $C>0$ such that the volume upper
bound 
\begin{equation*}
V_{f}\left( B_{p}\left( R\right) \right) \leq Ce^{\left( n-1+\alpha \right)
R}
\end{equation*}
holds for all $R>0.$
\end{lemma}

\begin{proof}[Proof of Lemma \protect\ref{Vol_Ric_Neg}]
As discussed above, we write $dV|_{\exp _{p}\left( r\xi \right) }=J\left(
r,\xi \right) drd\xi $ for $\xi \in S_{p}M.$ Let $J_{f}\left( r,\xi \right)
=e^{-f\left( r,\xi \right) }J\left( r,\xi \right) $ be the corresponding
weighted volume form. In the following, we will omit the dependence of these
quantities on $\xi .$ Along a minimizing geodesic starting from $p$, we have 
\begin{equation*}
\left( \frac{J^{\prime }}{J}\right) ^{\prime }\left( r\right) +\frac{1}{n-1}%
\left( \frac{J^{\prime }}{J}\right) ^{2}\left( r\right) +Ric\left( \frac{%
\partial }{\partial r},\frac{\partial }{\partial r}\right) \leq 0,
\end{equation*}%
where the differentiation is with respect to the $r$ variable. Integrating
this inequality from $1$ to $r$ and using the assumption that 
\begin{equation*}
Ric\left( \frac{\partial }{\partial r},\frac{\partial }{\partial r}\right)
+f^{\prime \prime }\left( r\right) \geq -\left( n-1\right) ,
\end{equation*}%
we get 
\begin{equation*}
\frac{J^{\prime }}{J}\left( r\right) +\frac{1}{n-1}\int_{1}^{r}\left( \frac{%
J^{\prime }}{J}\right) ^{2}\left( t\right) dt-f^{\prime }\left( r\right)
\leq \left( n-1\right) r+C_{0}
\end{equation*}%
for some constant $C_{0}>0$ independent of $r.$ Let us denote%
\begin{equation*}
u\left( t\right) :=\frac{J_{f}^{\prime }\left( t\right) }{J_{f}\left(
t\right) }=\frac{J^{\prime }}{J}\left( r\right) -f^{\prime }\left( r\right) .
\end{equation*}%
Then for any $r\geq 1,$ 
\begin{equation}
u\left( r\right) +\frac{1}{n-1}\int_{1}^{r}\left( u\left( t\right)
+f^{\prime }\left( t\right) \right) ^{2}dt\leq \left( n-1\right) r+C_{0}.
\label{v2}
\end{equation}%
The Cauchy-Schwarz inequality implies that 
\begin{equation*}
\int_{1}^{r}\left( u\left( t\right) +f^{\prime }\left( t\right) \right)
^{2}dt\geq \left( r-1\right) ^{-1}\left\{ \int_{1}^{r}\left( u\left(
t\right) +f^{\prime }\left( t\right) \right) dt\right\} ^{2}.
\end{equation*}%
Therefore, from (\ref{v2}) we obtain 
\begin{equation}
u\left( r\right) +\frac{1}{\left( n-1\right) r}\left( f\left( r\right)
-f\left( 1\right) +\int_{1}^{r}u\left( t\right) dt\right) ^{2}\leq \left(
n-1\right) r+C_{0}.  \label{v3}
\end{equation}%
We now claim that for any $r\geq 1,$ 
\begin{equation}
\int_{1}^{r}u\left( t\right) dt\leq \left( n-1+\alpha \right) r+\alpha
+2\beta +C_{0}.  \label{v4}
\end{equation}%
To prove this, define 
\begin{equation*}
v\left( r\right) :=\left( n-1+\alpha \right) r+\alpha +2\beta
+C_{0}-\int_{1}^{r}u\left( t\right) dt.
\end{equation*}%
We show instead that $v\left( r\right) \geq 0$ for all $r\geq 1.$ Clearly, $%
v\left( 1\right) >0.$ Suppose that $v$ does not remain positive for all $%
r\geq 1$ and let $R>1$ be the first number such that $v\left( R\right) =0.$
Then, 
\begin{equation*}
\int_{1}^{R}u\left( t\right) dt=\left( n-1+\alpha \right) R+\alpha +2\beta
+C_{0}.
\end{equation*}%
In other words, 
\begin{eqnarray*}
&&\frac{1}{\left( n-1\right) R}\left( f\left( R\right) -f\left( 1\right)
+\int_{1}^{R}u\left( t\right) dt\right) ^{2} \\
&=&\frac{1}{\left( n-1\right) R}\left( f\left( R\right) -f\left( 1\right)
+\left( n-1+\alpha \right) R+\alpha +2\beta +C_{0}\right) ^{2} \\
&\geq &\frac{1}{\left( n-1\right) R}\left( (n-1)R+C_{0}\right) ^{2}\geq
\left( n-1\right) R+2C_{0}.
\end{eqnarray*}%
Plugging this into (\ref{v3}), we conclude $u\left( R\right) \leq -C_{0}<0.$
This shows that $v^{\prime }\left( R\right) =\left( n-1+\alpha \right)
-u\left( R\right) >0,$ which implies the existence of a small enough $\delta
>0$ such that $v\left( R-\delta \right) <v\left( R\right) =0.$ This
obviously contradicts with the choice of $R.$

We have thus proved that (\ref{v4}) is true for any $r\geq 1,$ or 
\begin{equation*}
\log J_{f}\left( r\right) -\log J_{f}\left( 1\right) \leq \left( n-1+\alpha
\right) r+\alpha +2\beta +C_{0}.
\end{equation*}%
In particular, for $R\geq 1,$ we have the volume bound of the form 
\begin{equation*}
V_{f}\left( B_{p}\left( R\right) \right) \leq Ce^{\left( n-1+\alpha \right)
R},
\end{equation*}%
with the constant $C$ depending on $\alpha $, $\beta $ and $B_{p}\left(
1\right) .$
\end{proof}

We remark that in the special case of $f$ being bounded, hence $\alpha=0,$
the estimate becomes 
\begin{equation*}
V_{f}\left( B_{p}\left( R\right) \right) \leq Ce^{\left( n-1\right) R}
\end{equation*}
for $R\geq 0.$ This improves a result in \cite{W} in the sense that the rate
of exponential growth for the weighted volume does not depend on $%
\sup_{M}\left\vert f\right\vert.$

Lemma \ref{Vol_Ric_Neg} readily leads to the following estimate for the
bottom spectrum of $\Delta _{f}.$

\begin{theorem}
\label{Est_Ric_Neg} Let $\left( M,g,e^{-f}dv\right) $ be a complete smooth
metric measure space with $Ric_{f}\geq -\left( n-1\right).$ Assume the
linear growth rate of $f$ is $a.$ Then we have%
\begin{equation*}
\lambda _{1}\left( \Delta _{f}\right) \leq \frac{1}{4}\left( n-1+a\right)
^{2}.
\end{equation*}%
In particular, if $f$ is of sublinear growth, then the bottom spectrum of
the weighted Laplacian has the following sharp upper bound%
\begin{equation*}
\lambda _{1}\left( \Delta _{f}\right) \leq \frac{\left( n-1\right) ^{2}}{4}.
\end{equation*}
\end{theorem}

\begin{proof}[Proof of Theorem \protect\ref{Est_Ric_Neg}]
Let $\psi$ be a cut-off function on $B_{p}\left( R\right) $ such that $\psi
=1$ on $B_{p}\left( R-1\right) $ and $\left\vert \nabla \psi \right\vert
\leq 2.$ Set $\phi \left( y\right) :=e^{-\frac{\left( n-1+a+\varepsilon
\right) }{2}r\left( y\right) }\psi \left( y\right) $ as a test function in
the variational principle for $\lambda _{1}\left( \Delta _{f}\right),$ where 
$\varepsilon >0$ is an arbitrary positive constant. Then, by Lemma \ref%
{Vol_Ric_Neg}, we obtain 
\begin{equation*}
\lambda _{1}\left( \Delta _{f}\right) \leq \frac{\left( n-1+a+\varepsilon
\right) ^{2}}{4}.
\end{equation*}%
Since $\varepsilon $ is arbitrary, this implies $\lambda _{1}\left( \Delta
_{f}\right) \leq \frac{\left( n-1+a\right) ^{2}}{4}$.

In the case that $f$ is of sublinear growth, we can take $a=0.$ Therefore, $%
\lambda _{1}\left( \Delta _{f}\right) \leq \frac{1}{4}\left( n-1\right) ^{2}$%
and the Theorem is proved.
\end{proof}

\section{$\ f-$harmonic functions\label{f_harmonic}}

In this section, we establish the following gradient estimate for positive $%
f-$harmonic functions defined on $M.$

\begin{theorem}
\label{Grad_Est} Let $\left( M^{n},g,e^{-f}dv\right) $ be a complete smooth
metric measure space with $Ric_{f}\geq -\left( n-1\right) $. Assume that for
any $x\in M,$ 
\begin{equation*}
\sup_{y\in B_{x}\left( 1\right) }\left\vert f\left( y\right) -f\left(
x\right) \right\vert \leq a.
\end{equation*}%
Then there exists a constant $C\left( n,a\right) $ depending only on $n$ and 
$a$ such that for any $u>0$ with $\Delta _{f}u=0$ we have%
\begin{equation*}
\left\vert \nabla \log u\right\vert \leq C\left( n,a\right) .
\end{equation*}
\end{theorem}

Under the stronger assumption that $\left\vert \nabla f\right\vert \leq a,$
the result was proved in \cite{Wu} by essentially following Yau's classical
argument. However, it seems no longer possible to apply Yau's approach
directly once the hypothesis on $f$ only involves its oscillation on unit
balls. Note that the theorem in particular is applicable to the case $f$ is
bounded on $M.$ Our proof of Theorem \ref{Grad_Est} follows the strategy in 
\cite{MW}. We will first obtain local Neumann Poincar\'{e} and Sobolev
inequalities and then use the DeGiorgi-Nash-Moser theory. Let us first
recall the following Laplace comparison theorem from \cite{W},%
\begin{equation}
\Delta _{f}d\left( x,y\right) \leq \left( n-1\right) \coth r+\frac{2}{\sinh
^{2}r}\int_{0}^{r}\left( f\left( t\right) -f\left( r\right) \right) \cosh
\left( 2t\right) dt,  \label{Laplace_comp}
\end{equation}%
where $r:=d\left( x,y\right) ,$ $f\left( t\right) :=f\left( \gamma \left(
t\right) \right) $ and $\gamma \left( t\right) $ is a minimizing normal
geodesic such that $\gamma \left( 0\right) =x$ and $\gamma \left( r\right)
=y.$

Using the assumption on $f$ that $\left\vert f\left( t\right) -f\left(
r\right) \right\vert \leq a,$ we get 
\begin{equation*}
\Delta _{f}d\left( x,y\right) \leq \left( n-1+2a\right) \coth r
\end{equation*}%
for any $0<r<1.$ In particular, this yields 
\begin{eqnarray}
\frac{J_{f}\left( x,r_{2},\xi \right) }{J_{f}\left( x,r_{1},\xi \right) }
&\leq &\left( \frac{\sinh \left( r_{2}\right) }{\sinh \left( r_{1}\right) }%
\right) ^{n-1+2a},  \label{a1} \\
\frac{J\left( x,r_{2},\xi \right) }{J\left( x,r_{1},\xi \right) } &\leq
&e^{2a}\left( \frac{\sinh \left( r_{2}\right) }{\sinh \left( r_{1}\right) }%
\right) ^{n-1+2a}  \notag
\end{eqnarray}%
for any $0<r_{1}<r_{2}<1.$

Now the arguments in \cite{Bu} and \cite{HK} (see also \cite{MW} for the
case of smooth metric measure spaces) imply that we have the following local
Neumann Poincar\'{e} and Sobolev inequalities.

\begin{lemma}
\label{NP} Let $\left( M,g,e^{-f}dv\right) $ be a smooth metric measure
space of dimension $n$ with $Ric_{f}\geq -\left( n-1\right) .$ Assume that
for any $x\in M,$ 
\begin{equation*}
\sup_{y\in B_{x}\left( 1\right) }\left\vert f\left( y\right) -f\left(
x\right) \right\vert \leq a.
\end{equation*}%
Then for $x\in M$ and $0<r<1$ we have 
\begin{equation*}
\int_{B_{x}\left( r\right) }\left\vert \varphi -\varphi _{B_{x}\left(
r\right) }\right\vert ^{2}\leq C\cdot r^{2}\int_{B_{x}\left( r\right)
}\left\vert \nabla \varphi \right\vert ^{2}
\end{equation*}%
for any $\varphi \in C^{\infty }\left( B_{x}\left( r\right) \right) ,$ where 
$\varphi _{B_{x}\left( r\right) }:=V^{-1}\left( B_{x}\left( r\right) \right)
\int_{B_{x}\left( r\right) }\varphi $ and the constant $C$ depending only on
the dimension $n$ and $a.$
\end{lemma}

Note that the conclusion here that the constant $C$ is independent of $x$ is
stronger than that in \cite{MW}. This is due to the more restrictive
assumption on $f$ in Lemma \ref{NP}.

\begin{lemma}
\label{NS} Let $\left( M,g,e^{-f}dv\right) $ be a smooth metric measure
space of dimension $n$ with $Ric_{f}\geq -\left( n-1\right) .$ Assume that
for any $x\in M,$ 
\begin{equation*}
\sup_{y\in B_{x}\left( 1\right) }\left\vert f\left( y\right) -f\left(
x\right) \right\vert \leq a.
\end{equation*}%
Then there exist constants $\nu >2$ and $C$ depending only on $n$ and $a$
such that 
\begin{equation*}
\left( \int_{B_{x}\left( 1\right) }\left\vert \varphi -\varphi _{B_{x}\left(
1\right) }\right\vert ^{\frac{2\nu }{\nu -2}}\right) ^{\frac{\nu -2}{\nu }%
}\leq \frac{C}{V\left( B_{x}\left( 1\right) \right) ^{\frac{2}{\nu }}}%
\int_{B_{x}\left( 1\right) }\left\vert \nabla \varphi \right\vert ^{2}\ 
\end{equation*}%
$\ $for\ any$\ \varphi \in C^{\infty }\left( B_{x}\left( 1\right) \right) ,$
where $\varphi _{B_{x}\left( 1\right) }:=V^{-1}\left( B_{x}\left( 1\right)
\right) \int_{B_{x}\left( 1\right) }\varphi .$
\end{lemma}

Although we have stated the Poincar\'{e} and Sobolev inequalities in Lemma %
\ref{NP} and Lemma \ref{NS} in terms of the volume form $dv,$ we point out
that the same statements hold true with respect to $e^{-f}dv$ as well with
possibly a different $C.$ This is because the oscillation of $f$ on $%
B_{x}\left(1\right) $ is assumed to be uniformly bounded.

We are now ready to prove Theorem \ref{Grad_Est}. Our argument is a mixture
of both the Bochner identity and the DeGiorgi-Nash-Moser theory (see e.g. 
\cite{Li1, SC1}).

\begin{proof}[Proof of Theorem \protect\ref{Grad_Est}]
Let $u$ be a positive solution to $\Delta _{f}u=0.$ Then the Bochner formula
asserts that 
\begin{equation*}
\frac{1}{2}\Delta _{f}\left\vert \nabla u\right\vert ^{2}=\left\vert
u_{ij}\right\vert ^{2}+\left\langle \nabla \Delta _{f}u,\nabla
u\right\rangle +Ric_{f}\left( \nabla u,\nabla u\right) .
\end{equation*}%
Using the curvature lower bound, we get 
\begin{equation*}
\Delta _{f}\left\vert \nabla u\right\vert ^{2}\geq -2\left( n-1\right)
\left\vert \nabla u\right\vert ^{2}.
\end{equation*}%
In view of (\ref{a1}), Lemma \ref{NP} and Lemma \ref{NS}, we may apply the
Moser iteration argument (see \cite{Li1, SC1}) to $\left\vert\nabla
u\right\vert ^{2}$ to conclude that 
\begin{equation}
\sup_{B_{x}\left( \frac{1}{16}\right) }\left\vert \nabla u\right\vert
^{2}\leq \frac{C}{V_{f}\left( B_{x}\left( \frac{1}{8}\right) \right) }%
\int_{B_{x}\left( \frac{1}{8}\right) }\left\vert \nabla u\right\vert
^{2}e^{-f}  \label{a2}
\end{equation}%
for any $x\in M,$ where $C$ depends only on $n$ and $a.$

Now let $\phi $ be a cut-off function with support in $B_{x}\left( \frac{1}{4%
}\right)$ such that $\phi =1$ on $B_{x}\left( \frac{1}{8}\right) $ and $%
\left\vert \nabla \phi \right\vert \leq 16.$ Then, using $\Delta _{f}u=0,$
we have 
\begin{eqnarray*}
\int_{M}\left\vert \nabla u\right\vert ^{2}\phi ^{2}e^{-f}
&=&-2\int_{M}u\phi \left\langle \nabla u,\nabla \phi \right\rangle e^{-f} \\
&\leq &\frac{1}{2}\int_{M}\left\vert \nabla u\right\vert ^{2}\phi
^{2}e^{-f}+2\int_{M}u^{2}\left\vert \nabla \phi \right\vert ^{2}e^{-f}.
\end{eqnarray*}%
Therefore, 
\begin{equation*}
\int_{M}\left\vert \nabla u\right\vert ^{2}\phi ^{2}e^{-f} \leq
4\int_{M}u^{2}\left\vert \nabla \phi \right\vert ^{2}e^{-f}.
\end{equation*}%
In view of (\ref{a1}), we conclude 
\begin{gather*}
\frac{1}{V_{f}\left( B_{x}\left( \frac{1}{8}\right) \right) }%
\int_{B_{x}\left( \frac{1}{8}\right) }\left\vert \nabla u\right\vert
^{2}e^{-f}\leq \frac{c}{V_{f}\left( B_{x}\left( \frac{1}{8}\right) \right) }%
\int_{B_{x}\left( \frac{1}{4}\right) }u^{2}e^{-f} \\
\leq c\,\frac{V_{f}\left( B_{x}\left( \frac{1}{4}\right) \right) }{%
V_{f}\left( B_{x}\left( \frac{1}{8}\right) \right) }\,(\sup_{B_{x}\left( 
\frac{1}{4}\right) }u)^{2}\leq C(\sup_{B_{x}\left( \frac{1}{4}\right)
}u)^{2}.
\end{gather*}%
Combining with (\ref{a2}), we obtain 
\begin{equation}
\left\vert \nabla u\right\vert \left( x\right) \leq C\sup_{B_{x}\left( \frac{%
1}{4}\right) }u.  \label{a3}
\end{equation}

On the other hand, using (\ref{a1}), Lemma \ref{NP} and Lemma \ref{NS}, and
applying the Moser iteration argument to the equation $\Delta _{f}u=0,$ we
arrived at the following Harnack type inequality 
\begin{equation*}
\sup_{B_{x}\left( \frac{1}{4}\right) }u\leq C\inf_{B_{x}\left( \frac{1}{4}%
\right) }u,
\end{equation*}%
where $C$ is a constant depending only on $n$ and $a$. So we may rewrite (%
\ref{a3}) into 
\begin{equation*}
\left\vert \nabla u\right\vert \left( x\right) \leq C\left( n,a\right)
u\left( x\right) ,
\end{equation*}%
which is what we wanted to prove.
\end{proof}

\section{Rigidity \label{Rigid}}

In this section, we focus on the equality case of the estimate of the bottom
spectrum in Theorem \ref{Estimate} and prove the following rigidity theorem.

\begin{theorem}
\label{Rigid_Ric_Neg} Let $\left( M,g,e^{-f}dv\right) $ be a complete smooth
metric measure space of dimension $n\geq 3$ with $Ric_{f}\geq -\left(
n-1\right).$ Assume that $\left\vert \nabla f\right\vert \leq a$ on $M$ for
some constant $a\geq 0.$ If $\lambda _{1}\left( \Delta _{f}\right) =\frac{1}{%
4}\left( n-1+a\right) ^{2},$ then either $M$ is connected at infinity or $f$
is constant and $M$ is a warped product $M=\mathbb{R}\times N$ with $%
ds_{M}^{2}=dt^{2}+h^{2}\left( t\right) ds_{N}^{2},$ where $N$ is compact.
The function $h\left( t\right) =e^{t}$ if $n\geq 4$ and $h\left( t\right)
=e^{t}$ or $h\left( t\right) =\cosh t$ if $n=3.$
\end{theorem}

\begin{proof}[Proof of Theorem \protect\ref{Rigid_Ric_Neg}]
Assume that $M$ has at least two ends. We will divide our proof into two
cases according to the ends being $f-$nonparabolic or $f-$parabolic. Recall
that a manifold is called $f$-nonparabolic if $\Delta _{f}$ admits a
positive symmetric Green's function. Otherwise, it is called $f$-parabolic.
For an end of the manifold, the same definition applies, where now the
Green's function refers to the one satisfying the Neumann boundary
conditions.

We first deal with the case that there are at least two $f$-nonparabolic
ends. Then, according to a result in \cite{LT}, there is a bounded
nonconstant $f$-harmonic function $u$ on $M$ such that $\int_{M}\left\vert
\nabla u\right\vert ^{2}e^{-f}<\infty .$

By the Bochner formula, we have%
\begin{equation*}
\Delta _{f}|\nabla u|^{2}=2\left\vert u_{ij}\right\vert ^{2}+2Ric_{f}(\nabla
u,\nabla u).
\end{equation*}%
For each $x\in M,$ we may choose a local orthonormal frame $%
\{e_1,e_2,\cdots,e_n\}$ such that at $x$ we have $u_{1}=\left\vert \nabla
u\right\vert $ and $u_{i}=0$ if $i>1.$ Now a standard manipulation implies%
\begin{equation*}
\left\vert u_{ij}\right\vert ^{2}\geq \left\vert u_{11}\right\vert
^{2}+2\sum_{j=2}^{n}\left\vert u_{1j}\right\vert ^{2}+\frac{\left\vert
\Delta u-u_{11}\right\vert ^{2}}{n-1}.
\end{equation*}%
Since $\Delta _{f}u=0,$ we have%
\begin{eqnarray}
\left\vert \Delta u-u_{11}\right\vert ^{2} &=&\left\vert \left\langle \nabla
u,\nabla f\right\rangle -u_{11}\right\vert ^{2}\geq \left\vert
u_{11}\right\vert ^{2}-2\left\langle \nabla u,\nabla f\right\rangle
\left\vert u_{11}\right\vert  \label{z-1} \\
&\geq &\left\vert u_{11}\right\vert ^{2}-2a\left\vert \nabla u\right\vert
\left\vert u_{11}\right\vert  \notag
\end{eqnarray}%
as $\left\vert \nabla f\right\vert \leq a.$ Thus,%
\begin{equation*}
\left\vert u_{ij}\right\vert ^{2}\geq \frac{n}{n-1}\sum_{j=1}^{n}\left\vert
u_{1j}\right\vert ^{2}-\frac{2a|u_{11}||\nabla u|}{n-1}.
\end{equation*}%
Since the orthonormal frame is chosen that $e_{1}$ is in the direction of $%
\nabla u,$ it is easy to see%
\begin{equation*}
|\nabla |\nabla u||^{2}=\sum_{j=1}^{n}\left\vert u_{1j}\right\vert ^{2}
\end{equation*}%
and 
\begin{equation*}
|u_{11}|\leq |\nabla |\nabla u||.
\end{equation*}%
Therefore,%
\begin{equation*}
\left\vert u_{ij}\right\vert ^{2}\geq \frac{n}{n-1}|\nabla |\nabla u||^{2}-%
\frac{2a}{n-1}|\nabla |\nabla u|||\nabla u|.
\end{equation*}%
Together with the lower bound assumption on $Ric_{f},$ we obtain%
\begin{equation*}
|\nabla u|\Delta _{f}|\nabla u|\geq \frac{1}{n-1}|\nabla |\nabla u||^{2}-%
\frac{2a|\nabla |\nabla u|||\nabla u|}{n-1}-(n-1)|\nabla u|^{2}.
\end{equation*}%
Now let 
\begin{equation*}
\alpha :=\frac{n-2}{n-1}+\frac{\sqrt{n-2}\,a}{(n-1)^{2}}.
\end{equation*}%
Using the elementary inequality%
\begin{equation*}
2|\nabla |\nabla u|||\nabla u|\leq \frac{\sqrt{n-2}}{n-1}|\nabla |\nabla
u||^{2}+\frac{n-1}{\sqrt{n-2}}|\nabla u|^{2},
\end{equation*}%
we can rewrite the preceding inequality into%
\begin{equation}
\Delta _{f}|\nabla u|^{\alpha }\geq -\left( \sqrt{n-2}+\frac{a}{n-1}\right)
^{2}\,|\nabla u|^{\alpha }.  \label{z0}
\end{equation}

Now the argument in \cite{LW0} implies that $\lambda _{1}\left( \Delta
_{f}\right) \leq \left( \sqrt{n-2}+\frac{a}{n-1}\right) ^{2}.$ Moreover,
that the equality holds forces (\ref{z0}) into an equality also. In the case
of $n\ge 4,$ this contradicts with the assumption that $\lambda _{1}\left(
\Delta _{f}\right) =\frac{1}{4}\left( n-1+a\right) ^{2}.$ In the case $n=3,$
this indeed becomes an equality. So (\ref{z0}) and all the inequalities used
to prove (\ref{z0}) are equalities. In particular, from (\ref{z-1}) we
conclude $\left\langle \nabla f,\nabla u\right\rangle =0$ and $a\left\vert
\nabla u\right\vert \left\vert u_{11}\right\vert =0.$ Observe that we also
have $\left\vert u_{11}\right\vert =\left\vert \nabla \left\vert \nabla
u\right\vert \right\vert .$ Now if $a\neq 0,$ then we conclude $\left\vert
\nabla \left\vert \nabla u\right\vert \right\vert =0$ and $\left\vert \nabla
u\right\vert =C$ on $M.$ But this contradicts with $\int_{M}\left\vert
\nabla u\right\vert ^{2}e^{-f}<\infty $ as $\int_{M}e^{-f}=\infty$ by the
fact that $M$ is $f-$nonparabolic. Therefore, $a=0$ and $f$ is constant. So
we are back to the standard Laplacian case. By \cite{LW0}, $M=\mathbb{R}%
\times N$ with $ds_{M}^{2}=dt^{2}+\cosh ^{2}\left( t\right) ds_{N}^{2},$
where $N$ is compact.

We now focus on the case when the manifold has exactly one $f$-nonparabolic
end $E.$ If $M$ admits more than one end, then the end $F:=M\backslash E$
must be $f$-parabolic. Using the fact that $\lambda _{1}\left( \Delta
_{f}\right) =\frac{1}{4}\left( n-1+a\right) ^{2}$ and arguing as in \cite{LW}%
, we obtain 
\begin{equation}
V_{f}\left( F\backslash B_{p}\left( R\right) \right) \leq Ce^{-\left(
n-1+a\right) R}.  \label{z3}
\end{equation}%
Consider a ray $\gamma $ contained in the end $F$ and define the associated
Busemann function 
\begin{equation*}
\beta \left( x\right) :=\lim_{t\rightarrow \infty }\left( t-d\left( x,\gamma
\left( t\right) \right) \right) .
\end{equation*}%
Then, on the end $F,$ we have $\beta \left( x\right) \leq r\left( x\right)
+c,$ and on the end $E,$ $-r\left( x\right) -c\leq \beta \left( x\right)
\leq -r\left( x\right) +c$ by \cite{LW1}.

Denote by $\tau _{t}\left( s\right) $ the minimizing geodesic from $\gamma
\left( t\right) $ to $x$ that is parametrized by the arc length. According
to the Laplace comparison theorem in \cite{W}, we have%
\begin{equation*}
\Delta _{f}\left( d\left( x,\gamma \left( t\right) \right) \right) \leq
\left( n-1\right) \coth r-\frac{1}{\sinh ^{2}\left( r\right) }%
\int_{0}^{r}f^{\prime }\left( s\right) \sinh \left( 2s\right) ds,
\end{equation*}%
where $r:=d\left( x,\gamma \left( t\right) \right) $ and $f\left( s\right)
:=f\left( \tau _{t}\left( s\right) \right) .$ Since $\left\vert \nabla
f\right\vert \leq a,$ it is straightforward to see 
\begin{equation*}
\Delta _{f}\left( d\left( x,\gamma \left( t\right) \right) \right) \leq
\left( n-1\right) \coth r+a.
\end{equation*}
By the definition of the Busemann function, it is now standard to verify
that the following estimate holds in the sense of distributions.%
\begin{equation}
\Delta _{f}\beta \left( x\right) \geq -\left( n-1+a\right) .  \label{z4}
\end{equation}

Note by the Laplacian comparison theorem that $\Delta _{f}r\leq \left(
n-1\right) \coth r+a,$ we have 
\begin{equation}
V_{f}\left( B_{p}\left( R\right) \cap E\right) \leq Ce^{\left( n-1+a\right)
R}  \label{z5}
\end{equation}
for all $R>0.$

Consider the function 
\begin{equation*}
B:=e^{\frac{1}{2}\left( n-1+a\right) \beta }.
\end{equation*}%
Using (\ref{z4}) and the fact that $\left\vert \nabla \beta \right\vert =1,$
we conclude 
\begin{equation}
\Delta _{f}B\geq -\frac{1}{4}\left( n-1+a\right) ^{2}B.  \label{z6}
\end{equation}

Let $\phi$ be a cut-off function with support in $B_{p}\left( 2R\right) $
such that $\phi =1$ on $B_{p}\left( R\right) $ and $\left\vert \nabla \phi
\right\vert \leq \frac{C}{R}.$ Then, 
\begin{gather*}
\frac{1}{4}\left( n-1+a\right) ^{2}\int_{M}\left( B\phi \right)
^{2}e^{-f}+\int_{M}B\left( \Delta _{f}B\right) \phi ^{2}e^{-f} \\
\leq \int_{M}\left\vert \nabla \left( B\phi \right) \right\vert
^{2}e^{-f}+\int_{M}B\left( \Delta _{f}B\right) \phi ^{2}e^{-f} \\
=\int_{M}\left\vert \nabla \phi \right\vert ^{2}B^{2}e^{-f} \\
\leq \frac{C}{R},
\end{gather*}%
where we have used (\ref{z3}) and (\ref{z5}) in the last inequality. Letting 
$R$ go to infinity and taking into account of (\ref{z6}), we conclude $%
\Delta _{f}B=-\frac{1}{4}\left( n-1+a\right) ^{2}B.$ Equivalently, 
\begin{equation*}
\Delta _{f}\beta =-\left( n-1+a\right) \ \ \text{and\ }\ \left\vert \nabla
\beta \right\vert =1
\end{equation*}%
everywhere on $M.$ So the Bochner formula implies that 
\begin{eqnarray}
0 &=&\frac{1}{2}\Delta _{f}\left\vert \nabla \beta \right\vert
^{2}=\left\vert \beta _{ij}\right\vert ^{2}+\left\langle \nabla \Delta
_{f}\beta ,\nabla \beta \right\rangle +Ric_{f}\left( \nabla \beta ,\nabla
\beta \right)  \label{z7} \\
&\geq &\left\vert \beta _{ij}\right\vert ^{2}-\left( n-1\right) .  \notag
\end{eqnarray}%
On the other hand, under an orthonormal frame $\{e_1,e_2,\cdots,e_n\}$ so
that $\beta _{1}=\left\vert \nabla \beta \right\vert =1$ and $\beta _{i}=0$
for $i>1,$ one has $\beta _{11}=0.$ In particular, 
\begin{equation}
\left\vert \beta _{ij}\right\vert ^{2}\geq \frac{1}{n-1}\left( \Delta \beta
\right) ^{2}=\frac{1}{n-1}\left( n-1+a-\left\langle \nabla f,\nabla \beta
\right\rangle \right) ^{2}.  \label{z8}
\end{equation}%
Since $\left\vert \left\langle \nabla f,\nabla \beta \right\rangle
\right\vert \leq a,$ clearly $\left\vert \beta _{ij}\right\vert ^{2}\geq
n-1. $ In conclusion, both (\ref{z7}) and (\ref{z8}) must be equalities.
Reading from the equality case of (\ref{z8}), we assert that $\left( \beta
_{ij}\right) $ is a diagonal matrix. Moreover, $\left\langle \nabla f,\nabla
\beta \right\rangle =a,$ which implies $\nabla f=a\nabla \beta $ or $%
f=a\beta $ up to a constant. Using again that $\Delta _{f}\beta =-\left(
n-1+a\right) $ and $\left\langle \nabla f,\nabla \beta \right\rangle =a,$ we
deduce that $\Delta \beta =-\left( n-1\right) .$ Hence, the diagonal entries
of $\left( \beta _{ij}\right) $ are given by $\beta _{11}=0$ and $\beta
_{ii}=-1$ for $i\geq 2.$ This information on $\left( \beta _{ij}\right) $
together with the fact that $\left\vert \nabla \beta \right\vert =1$ leads
to the splitting of $M$ as a warped product $\mathbb{R}\times N$ with $%
ds_{M}^{2}=dt^{2}+e^{-2t}ds_{N}^{2}.$ The manifold $N$ is given by the level
set of the Busemann function $\beta ^{-1}\left( 0\right) =\left\{ x:\beta
\left( x\right) =0\right\}.$ The splitting line is given by the integral
curves of $\nabla \beta .$ The manifold $N$ is necessarily compact due to
the fact that $M$ is assumed to have (at least) two ends. For more details,
see \cite{LW}.

Now we show that in fact $f$ has to be constant in this case. Indeed,
according to a standard computation of the curvature of a warped product
metric, we have $Ric_{ij}=Ric_{ij}^{N}-\left( n-1\right) g_{ij}$ for $%
i,j\geq 2,$ where $Ric^{N}$ is the Ricci curvature of $N.$ The condition
that $Ric_{f}\geq -\left( n-1\right) $ is then equivalent to $Ric^{N}\geq
ae^{-2t}$ on $N.$ Since $a\geq 0$ and $t\in \mathbb{R}$ is arbitrary, this
is impossible due to the compactness of $N$ unless $a=0.$ This proves the
Theorem.
\end{proof}

\section{Ends of expanding Ricci solitons \label{Expanding}}

In this section, we investigate the issue of whether an expanding gradient
Ricci soliton is necessarily connected at infinity. Recall that an expanding
gradient Ricci soliton is a Riemannian manifold $\left(M,g\right) $ such
that $Ric+Hess\left( f\right) =-\frac{1}{2}g$ for some function $f.$ Such $f$
is called the potential function of the soliton.

We begin by collecting some basic properties of expanding gradient Ricci
solitons. First, it is known that 
\begin{equation}
S+\left\vert \nabla f\right\vert ^{2}=-f  \label{e1}
\end{equation}%
after adding a suitable constant to $f,$ where $S$ denotes the scalar
curvature of $M.$ Also, taking trace of the soliton equation, we obtain 
\begin{equation}
\Delta f+S=-\frac{n}{2}.  \label{e2}
\end{equation}%
On the other hand, by the maximum principle, it was proved in \cite{PRS, Z}
that 
\begin{equation}
S\geq -\frac{n}{2}.  \label{e3}
\end{equation}%
Moreover, if $S=-\frac{n}{2}$ at some point, then the manifold must be
Einstein and the potential function $f$ is constant. Such soliton is called
a trivial one.

Some elementary examples of expanding gradient Ricci solitons include $M=%
\mathbb{R}^{k}\times N^{n-k},$ where $N^{n-k}$ is an Einstein manifold with $%
Ric_{N}=-\frac{1}{2}g_{N}$ and $\mathbb{R}^{k}$ the Gaussian expanding Ricci
soliton with potential function $f=-\frac{1}{4}\left\vert x\right\vert ^{2}.$

From (\ref{e1}) and (\ref{e3}), it is easy to see that $\left( -f\right) $
grows at most quadratically. Indeed, 
\begin{equation*}
\left( -f\right) \left( x\right) \leq \frac{1}{4}r^{2}\left( x\right)
+cr\left( x\right) .
\end{equation*}

In view of this upper bound, it is natural to look for a matching lower
bound for $-f,$ which has been achieved in the case of shrinking gradient
Ricci solitons \cite{CZ}. However, in contrast to shrinking gradient Ricci
solitons, in general such a pointwise lower bound is not to be expected for
expanding gradient Ricci solitons. Indeed, for the preceding examples of the
form $M=\mathbb{R}^{n-k}\times N^{k},$ the potential function is given by $%
f(x,y)=-\frac{1}{4}\left\vert x\right\vert ^{2}$ for $x\in \mathbb{R}^{n-k}$
and $y\in N.$ If we take $N^{k}$ to be the simply connected hyperbolic space
of Ricci curvature $-\frac{1}{2},$ then $N$ is noncompact and the potential
function $f$ does not satisfy the desired bounds. Nonetheless, we have the
following estimate concerning the potential function $f.$ The result in
particular implies that an expanding gradient Ricci soliton must be trivial
if its potential function is of subquadratic growth. Prior to our result, it
was known from \cite{PRS} that if $\left\vert \nabla f\right\vert $ is
bounded on $M,$ then $M$ is Einstein.

\begin{theorem}
\label{Potential} Let $\left( M,g,f\right) $ be a nontrivial complete
expanding gradient Ricci soliton. Then there exists constant $C$ such that 
\begin{equation*}
\frac{1}{4}\,r^{2}-C\,r^{\frac{3}{2}}\sqrt{\,\ln r}\leq \sup_{\partial
B_{p}\left( r\right) }(-f)\left( x\right) \leq \frac{1}{4}r^{2}+Cr
\end{equation*}%
for all $r>2.$
\end{theorem}

\begin{proof}[Proof of Theorem \protect\ref{Potential}]
As indicated above, we have 
\begin{eqnarray}
S+\Delta f &=&-\frac{n}{2}  \label{p1} \\
\left\vert \nabla f\right\vert ^{2}+S &=&-f  \notag \\
S &>&-\frac{n}{2}.  \notag
\end{eqnarray}%
Now the upper bound readily follows from (\ref{p1}). So we only need to
prove the lower bound.

Let us denote 
\begin{equation*}
u:=\frac{n}{2}-f=\frac{n}{2}+S+\left\vert \nabla f\right\vert ^{2}>0.
\end{equation*}%
For $k>0,$ we compute 
\begin{equation}
\Delta e^{2k\sqrt{u}}=\left( \frac{k}{\sqrt{u}}\Delta u+\left( \frac{k^{2}}{u%
}-\frac{k}{2u\sqrt{u}}\right) \left\vert \nabla u\right\vert ^{2}\right)
e^{2k\sqrt{u}}.  \label{p2}
\end{equation}%
From (\ref{p1}), we have 
\begin{eqnarray}
\Delta u &=&\left( \frac{n}{2}+S\right) \ \ \ \text{and}  \label{p3} \\
\left\vert \nabla u\right\vert ^{2} &=&u-\left( \frac{n}{2}+S\right) . 
\notag
\end{eqnarray}%
It follows that%
\begin{equation}
\Delta e^{2k\sqrt{u}}=k\left\{ k-\frac{1}{2\sqrt{u}}+\left( \frac{n}{2}%
+S\right) \left( \frac{1}{\sqrt{u}}+\frac{1}{2u\sqrt{u}}-\frac{k}{u}\right)
\right\} e^{2k\sqrt{u}}.  \label{p4}
\end{equation}%
Multiplying (\ref{p4}) by $u^{k^{2}}$ and integrating by parts on $%
B_{p}\left( r\right) ,$ we have 
\begin{gather*}
\int_{B_{p}\left( r\right) }u^{k^{2}}\left( \Delta e^{2k\sqrt{u}}\right)
=-\int_{B_{p}\left( r\right) }\left\langle \nabla e^{2k\sqrt{u}},\nabla
u^{k^{2}}\right\rangle +\int_{\partial B_{p}\left( r\right) }u^{k^{2}}\frac{%
\partial }{\partial r}\left( e^{2k\sqrt{u}}\right) \\
\leq -k^{3}\int_{B_{p}\left( r\right) }\left\vert \nabla u\right\vert
^{2}u^{k^{2}-\frac{3}{2}}e^{2k\sqrt{u}}+k\int_{\partial B_{p}\left( r\right)
}\frac{1}{\sqrt{u}}u^{k^{2}}\left\vert \nabla u\right\vert e^{2k\sqrt{u}} \\
=-k^{3}\int_{B_{p}\left( r\right) }\frac{1}{\sqrt{u}}u^{k^{2}}e^{2k\sqrt{u}%
}+k^{3}\int_{B_{p}\left( r\right) }\left( \frac{n}{2}+S\right) \frac{1}{u%
\sqrt{u}}u^{k^{2}}e^{2k\sqrt{u}} \\
+k\int_{\partial B_{p}\left( r\right) }\frac{1}{\sqrt{u}}u^{k^{2}}\left\vert
\nabla u\right\vert e^{2k\sqrt{u}},
\end{gather*}%
where in the last line we have used (\ref{p3}). Consequently, from (\ref{p4}%
) it follows%
\begin{gather}
\int_{\partial B_{p}\left( r\right) }\frac{\left\vert \nabla u\right\vert }{%
\sqrt{u}}u^{k^{2}}e^{2k\sqrt{u}}  \label{p5} \\
\geq \int_{B_{p}\left( r\right) }\left\{ k+\frac{1}{\sqrt{u}}\left( k^{2}-%
\frac{1}{2}\right) +\left( \frac{n}{2}+S\right) \left( \frac{1}{\sqrt{u}}-%
\frac{1}{u\sqrt{u}}\left( k^{2}-\frac{1}{2}\right) -\frac{k}{u}\right)
\right\} u^{k^{2}}e^{2k\sqrt{u}}.  \notag
\end{gather}%
We now claim that 
\begin{equation}
k+\frac{1}{\sqrt{u}}\left( k^{2}-\frac{1}{2}\right) +\left( \frac{n}{2}%
+S\right) \left( \frac{1}{\sqrt{u}}-\frac{1}{u\sqrt{u}}\left( k^{2}-\frac{1}{%
2}\right) -\frac{k}{u}\right) \geq k\frac{\left\vert \nabla u\right\vert }{%
\sqrt{u}}.  \label{p6}
\end{equation}

We prove this directly by checking it at arbitrary point $x\in M.$ Let us
denote for simplicity 
\begin{equation*}
\alpha :=\frac{n}{2}+S\left( x\right) ,
\end{equation*}%
and let 
\begin{equation*}
\gamma :=\sqrt{\frac{u\left( x\right) }{\alpha }}\geq 1.
\end{equation*}%
The fact that $\gamma \geq 1$ follows from (\ref{p3}). Notice that (\ref{p6}%
) is equivalent to%
\begin{equation*}
k\sqrt{u}+\left( k^{2}-\frac{1}{2}\right) +\alpha \left( 1-\frac{1}{u}\left(
k^{2}-\frac{1}{2}\right) -\frac{k}{\sqrt{u}}\right) \geq k\sqrt{u-\alpha }.
\end{equation*}%
This inequality is rewritten into the following equivalent form after
replacing $u$ in terms of $\gamma $ and rearranging the terms. 
\begin{equation}
\alpha -k\sqrt{\alpha }\left( \frac{1}{\gamma }-\gamma +\sqrt{\gamma ^{2}-1}%
\right) +\left( k^{2}-\frac{1}{2}\right) \left( 1-\frac{1}{\gamma ^{2}}%
\right) \geq 0.  \label{p7}
\end{equation}%
The discriminant of this quadratic inequality in $\sqrt{\alpha }$ is given
by 
\begin{eqnarray*}
D &:&=k^{2}\left( \frac{1}{\gamma }-\gamma +\sqrt{\gamma ^{2}-1}\right)
^{2}-4\left( k^{2}-\frac{1}{2}\right) \left( 1-\frac{1}{\gamma ^{2}}\right)
\\
&=&\left( 1-\frac{1}{\gamma ^{2}}\right) \left\{ k^{2}\left( \gamma -\sqrt{%
\gamma ^{2}-1}\right) ^{2}-4\left( k^{2}-\frac{1}{2}\right) \right\} .
\end{eqnarray*}%
Since 
\begin{equation*}
0\leq \gamma -\sqrt{\gamma ^{2}-1}\leq 1\ \ \ \text{and \ }\gamma \geq 1,
\end{equation*}%
it follows that for $k\geq 1,$ 
\begin{equation*}
D \leq \left( 1-\frac{1}{\gamma ^{2}}\right) \left\{ -3k^{2}+2\right\} \leq
0.
\end{equation*}%
This proves that (\ref{p7}) is true for any $\gamma \geq 1$ and for any $%
\alpha \geq 0.$ Therefore, (\ref{p6}) holds true at any $x\in M.$

From (\ref{p5}) and (\ref{p6}) we get that 
\begin{equation*}
k\,\int_{B_{p}\left( r\right) }\left\vert \nabla u\right\vert u^{k^{2}-\frac{%
1}{2}}e^{2k\sqrt{u}}\leq \int_{\partial B_{p}\left( r\right) }\left\vert
\nabla u\right\vert u^{k^{2}-\frac{1}{2}}e^{2k\sqrt{u}}.
\end{equation*}%
Hence the function 
\begin{equation*}
w\left( r\right) :=\int_{B_{p}\left( r\right) }\left\vert \nabla
u\right\vert u^{k^{2}-\frac{1}{2}}e^{2k\sqrt{u}}
\end{equation*}%
satisfies $kw\left( r\right) \leq w^{\prime }\left( r\right) $ for any $%
r\geq 0.$ Since $(M,g,f)$ is assumed to be a nontrivial Ricci soliton, there
exists a positive radius $\ r_{0}$ for which $w\left( r_{0}\right) >0.$
Integrating $kw\left( t\right) \leq w^{\prime }\left( t\right) $ from $%
t=r_{0}$ to $t=r,$ we conclude that there exists a positive constant $C>0$
such that $w\left( r\right) \geq C\,e^{kr}.$ Therefore, we have proved that 
\begin{equation}
\int_{\partial B_{p}\left( r\right) }\left\vert \nabla u\right\vert u^{k^{2}-%
\frac{1}{2}}e^{2k\sqrt{u}}\geq Ce^{kr}\ \ \ \text{for \ }r\geq r_{0}.
\label{p8}
\end{equation}%
We now prove that there exists a constant $c\left( n\right) $ depending only
on $n$ such that 
\begin{equation}
A\left( \partial B_{p}\left( r\right) \right) \leq Ce^{c\left( n\right) r}\
\ \ \ \text{for\ \ \ }r\geq r_{0}.  \label{p9}
\end{equation}%
Let us stress that (\ref{p9}) refers to the usual area, not the weighted
one. This claim follows as in Lemma 2.1. Indeed, from Lemma 2.1 we have%
\begin{equation*}
\frac{J^{\prime }}{J}\left( r\right) +\frac{1}{n-1}\int_{1}^{r}\left( \frac{%
J^{\prime }}{J}\right) ^{2}\left( t\right) dt\leq f^{\prime }\left( r\right)
+\frac{1}{2}r+C_{0}.
\end{equation*}%
Since 
\begin{equation*}
\sup_{B_{p}\left( r\right) }\left\vert \nabla f\right\vert \leq \frac{1}{2}%
r+c,
\end{equation*}%
it follows that 
\begin{equation*}
\frac{J^{\prime }}{J}\left( r\right) +\frac{1}{\left( n-1\right) r}\left(
\int_{1}^{r}\frac{J^{\prime }}{J}\left( t\right) dt\right) ^{2}\leq r+C_{0}.
\end{equation*}%
The argument in Lemma 2.1 now shows that 
\begin{equation*}
A\left( \partial B_{p}\left( r\right) \right) \leq Ce^{2\sqrt{n-1}r}.
\end{equation*}%
This proves (\ref{p9}) is true. Plugging this into (\ref{p8}) and using that 
\begin{equation*}
\left\vert \nabla u\right\vert ^{2}\leq u\leq \frac{1}{4}r^{2}+cr,
\end{equation*}%
we have 
\begin{equation*}
\sup_{\partial B_{p}\left( r\right) }e^{2k\sqrt{u}}\geq Ce^{kr-2k^{2}\ln
r-c\left( n\right) r}.
\end{equation*}%
In other words, 
\begin{equation*}
\sup_{\partial B_{p}\left( r\right) }2\sqrt{u}\geq r-2k\ln r-c\,\frac{r}{k}.
\end{equation*}%
Since this estimate is true for each fixed $r$ over all $k,$ we may optimize
by choosing $k=\sqrt{\frac{r}{\ln r}}.$ It is easy to see that this proves
the Theorem.
\end{proof}

Theorem \ref{Potential} shows that the results we proved in Section \ref%
{Rigid} cannot be applied directly to expanding gradient Ricci solitons as
the boundedness assumption on $\left\vert \nabla f\right\vert $ is not
available. To address the issue of connectedness at infinity, we have to
proceed somewhat differently to obtain a $\lambda _{1}$ estimate. Also, our
proof to rule out the existence of small ends seems to rely on some specific
properties of expanding gradient Ricci solitons.

\begin{theorem}
\label{Connected}Let $\left( M,g,f\right) $ be a complete gradient expanding
Ricci soliton. Assume that $S\geq -\frac{n-1}{2}$ on $M.$ Then either $M$ is
connected at infinity or $M=\mathbb{R}\times N^{n-1},$ where $N$ is a
compact Einstein manifold and $\mathbb{R}$ is the Gaussian expanding soliton.
\end{theorem}

Before proving the Theorem, we first establish a weighted Poincar\'{e}
inequality for expanding gradient Ricci solitons. The importance of a
weighted Poincar\'{e} inequality for the issue of connectedness at infinity
of Riemannian manifolds has been exemplified in \cite{LW2}.

\begin{lemma}
\label{WPI}Let $\left( M,g,f\right) $ be a complete nontrivial expanding
gradient Ricci soliton. Define $\sigma :=S+\frac{n}{2}.$ Then $\sigma >0$ on 
$M$ and 
\begin{equation*}
\int_{M}\sigma \phi ^{2}e^{-f}\leq \int_{M}\left\vert \nabla \phi
\right\vert ^{2}e^{-f}
\end{equation*}%
for any $\phi \in C_{0}^{\infty }\left( M\right) .$
\end{lemma}

\begin{proof}[Proof of Lemma \protect\ref{WPI}]
The fact that $\sigma >0$ is clear from (\ref{e3}). Using (\ref{e1}) and (%
\ref{e2}), we compute

\begin{equation*}
\Delta _{f}e^{f}=\left( \Delta _{f}\left( f\right) +\left\vert \nabla
f\right\vert ^{2}\right) e^{f}=\left( \Delta f\right) e^{f}=-\left( \frac{n}{%
2}+S\right) e^{f}.
\end{equation*}%
Now Proposition 1.1 in \cite{LW2} implies the claimed weighted Poincar\'{e}
inequality.
\end{proof}

Using Lemma \ref{WPI}, we proceed to prove the Theorem.

\begin{proof}[Proof of Theorem \protect\ref{Connected}]
Since $M$ satisfies a weighted Poincar\'{e} inequality, from \cite{LW2} it
follows that $M$ is $f-$nonparabolic. By hypothesis $S\geq -\frac{n-1}{2},$
we see that 
\begin{equation*}
\sigma =\frac{n}{2}+S\geq \frac{1}{2},
\end{equation*}%
i.e., the bottom spectrum of the weighted Laplacian on $M$ satisfies $%
\lambda _{1}\left( \Delta _{f}\right) \geq \frac{1}{2}.$

We first show that all ends of $M$ must be $f-$nonparabolic. Suppose $E$ is
an $f-$parabolic end of $M.$ Let us observe first that both $\left\vert
f\right\vert $ and $\left\vert \nabla f\right\vert $ must be bounded on $E.$
To see this, let us denote by 
\begin{equation*}
u:=n-2f.
\end{equation*}%
It is easy to check, using (\ref{p1}), that $u$ has the following properties
on $M:$%
\begin{eqnarray*}
u &\geq &1 \\
\Delta _{f}u &=&u \\
\left\vert \nabla u\right\vert ^{2} &\leq &2u.
\end{eqnarray*}%
Consequently, a direct computation shows that the function $w:=e^{-\frac{1}{2} u}>0$
verifies $\Delta _{f}w\leq 0.$ If $u$ is an unbounded function on $E,$ then $%
w$ is a positive $f-$superharmonic function on $E,$ which achieves its
infimum at the infinity of $E.$ It is well known that this implies that $E$
is $f-$nonparabolic, see \cite{L} for details. This contradicts our
assumption that $E$ is $f-$parabolic, therefore $u$ must be bounded on $E.$
In particular, there exists a constant $A>0$ such that 
\begin{equation}
\left\vert f\right\vert +\left\vert \nabla f\right\vert \leq A\ \ \ \text{on 
}E.  \label{c5}
\end{equation}

By the Laplace comparison theorem in \cite{W}, or cf. (\ref{Laplace_comp})
here, it follows that for any two points $x,y\in E,$ 
\begin{equation*}
\Delta _{f}d\left( x,y\right) \leq a\coth d\left( x,y\right) ,
\end{equation*}%
for some constant $a$ independent of $x$ or $y.$ It is standard to obtain
from here that 
\begin{equation}
V_{f}\left( B_{x}\left( 1\right) \right) \geq C_{1}e^{-C_{2}r\left( x\right)
},  \label{c3}
\end{equation}%
for $C_{1}$ and $C_{2}$ independent of $x.$ Here, $r(x):=d(p,x)$, for $p \in M$ fixed. 

In the following, we obtain a contradiction to (\ref{c3}). Since $\sigma
\geq \frac{1}{2},$ it follows by \cite{LW0} that 
\begin{equation}
\int_{E\backslash E\left( r\right) }e^{-f}\leq Ce^{-\sqrt{2}r},  \label{c1}
\end{equation}%
where $E\left( r\right) :=E\cap B_{p}\left( r\right) $ and $C$ is a constant
independent of $r.$ Since $\Delta _{f}u=u,$ \ it is easy to see that 
\begin{equation*}
\Delta _{f}u^{k}\geq ku^{k}\ \ \ \text{for all }k\geq 1.
\end{equation*}%
We claim that Li-Wang's result in \cite{LW0} implies 
\begin{equation}
\int_{E\backslash E\left( r\right) }u^{2k}e^{-f}\leq e^{-\sqrt{4k+2}\,\left(
r-r_{0}\right) }\int_{E\left( r_{0}\right) \backslash E\left( r_{0}-1\right)
}u^{2k}e^{-f}  \label{c2}
\end{equation}%
for any $r>2r_{0}$ and any positive integer $k.$ For that, we need only to
check that 
\begin{equation*}
\frac{1}{R}\int_{E\left( R\right) }u^{2k}e^{-\sqrt{4k+2}\,r}e^{-f}%
\rightarrow 0\ \ \text{as \ }R\rightarrow \infty .
\end{equation*}%
This is clearly implied by the fact that $u$ is bounded on $E$ and by (\ref{c1}). 
Hence, the claim (\ref{c2}) is true. Since $u\geq 1$ is bounded, we see from
(\ref{c2}) that for any $k\geq 1,$ there exists a constant $C\left( k\right) 
$ so that 
\begin{equation*}
\int_{E\backslash E\left( r\right) }e^{-f}\leq C\left( k\right) e^{-\sqrt{%
4k+2}\,r},\ \ \ \text{for all }r>2r_{0}.
\end{equation*}

Choosing $k$ large enough, this clearly contradicts with (\ref{c3}). In
conclusion, all ends of $M$ must be $f-$nonparabolic.

We now deal with the $f-$nonparabolic ends. Suppose $M$ has at least two
ends. Let $E$ be an $f-$nonparabolic end. Then it follows that $%
F:=M\backslash E$ is also an $f-$nonparabolic end. By \cite{LT}, there
exists an $f-$harmonic function $h$ on $M$ such that $0<h<1$ and $%
\int_{M}\left\vert \nabla h\right\vert ^{2}e^{-f}<\infty .$

Now the Bochner formula together with the Kato inequality shows%
\begin{equation*}
\frac{1}{2}\Delta _{f}\left\vert \nabla h\right\vert ^{2}=\left\vert
h_{ij}\right\vert ^{2}-\frac{1}{2}\left\vert \nabla h\right\vert ^{2}\geq
\left\vert \nabla \left\vert \nabla h\right\vert \right\vert ^{2}-\frac{1}{2}%
\left\vert \nabla h\right\vert ^{2}
\end{equation*}%
or 
\begin{equation*}
\Delta _{f}\left\vert \nabla h\right\vert \geq -\frac{1}{2}\left\vert \nabla
h\right\vert .
\end{equation*}%
We use this in the weighted Poincar\'{e} inequality to find that for any
cut-off function $\phi ,$ 
\begin{gather*}
\int_{M}\sigma \left\vert \nabla h\right\vert ^{2}\phi ^{2}e^{-f}\leq
\int_{M}\left\vert \nabla \left( \left\vert \nabla h\right\vert \phi \right)
\right\vert ^{2}e^{-f} \\
=\int_{M}\left\vert \nabla \left\vert \nabla h\right\vert ^{2}\right\vert
\phi ^{2}e^{-f}+\frac{1}{2}\int_{M}\left\langle \nabla \left\vert \nabla
h\right\vert ^{2},\nabla \phi ^{2}\right\rangle e^{-f}+\int_{M}\left\vert
\nabla \phi \right\vert ^{2}\left\vert \nabla h\right\vert ^{2}e^{-f} \\
=-\int_{M}\left\vert \nabla h\right\vert \left( \Delta _{f}\left\vert \nabla
h\right\vert \right) \phi ^{2}e^{-f}+\int_{M}\left\vert \nabla \phi
\right\vert ^{2}\left\vert \nabla h\right\vert ^{2}e^{-f} \\
\leq \frac{1}{2}\int_{M}\left\vert \nabla h\right\vert ^{2}\phi
^{2}e^{-f}+\int_{M}\left\vert \nabla \phi \right\vert ^{2}\left\vert \nabla
h\right\vert ^{2}e^{-f}.
\end{gather*}%
Notice that we may choose $\phi $ so that 
\begin{equation*}
\int_{M}\left\vert \nabla \phi \right\vert ^{2}\left\vert \nabla
h\right\vert ^{2}e^{-f}\rightarrow 0\ .
\end{equation*}%
Since $\sigma \geq \frac{1}{2},$ this implies $\sigma =\frac{1}{2}$ or $S=-%
\frac{n}{2}+\frac{1}{2}$ on $M.$ The rigidity of gradient Ricci solitons
with constant scalar curvature has been studied in \cite{PW}. Our situation
here is more special since we have a specific value for $S.$ Let us denote 
\begin{equation*}
v:=2\sqrt{-f+\frac{n-1}{2}}=2\left\vert \nabla f\right\vert .
\end{equation*}%
Notice that $\left\vert \nabla v\right\vert =1$ at points
where $v \neq 0.$ Moreover, observe that since%
\begin{eqnarray*}
\Delta _{f}v^{2} &=&4\Delta _{f}\left( -f\right) =4\left( \frac{n}{2}%
-f\right) =v^{2}+2 \\
\Delta _{f}v^{2} &=&2v\Delta _{f}v+2\left\vert \nabla v\right\vert
^{2}=2v\Delta _{f}v+2,
\end{eqnarray*}%
we get that $\Delta _{f}v=\frac{1}{2}v$ whenever $v\neq 0.$ This implies
that $v$ is in fact smooth and $\left\vert \nabla v\right\vert =1$
everywhere on $M.$ The Bochner formula gives 
\begin{equation*}
0=\frac{1}{2}\Delta _{f}\left\vert \nabla v\right\vert ^{2}=\left\vert
v_{ij}\right\vert ^{2}-\frac{1}{2}\left\vert \nabla v\right\vert
^{2}+\left\langle \nabla \Delta _{f}v,\nabla v\right\rangle =\left\vert
v_{ij}\right\vert ^{2}.
\end{equation*}%
Therefore, $v_{ij}=0$ and $M$ admits a parallel vector field $\nabla v.$ So $%
M$ is isometric to $\mathbb{R}\times N.$ Since $M$ is assumed to have two
ends, $N$ must be a compact expanding gradient Ricci soliton. But it is well
known that $N$ then has to be Einstein. The Theorem is proved.
\end{proof}

\section{\protect\bigskip Volume of shrinking Ricci solitons\label{Shrinkers}%
}

Consider a shrinking gradient Ricci soliton $\left( M,g,f\right) .$ By
definition, this is a Riemannian manifold $\left( M,g\right) $ for which
there exists a smooth $f$ so that $Ric+Hess\left( f\right) =\frac{1}{2}g.$
Taking the trace of this equation, we also obtain $S+\Delta f=\frac{n}{2},$
where $S$ denotes the scalar curvature of $M.$ It is well known \cite{H}
that $S+\left\vert \nabla f\right\vert ^{2}=f$ after adjusting $f$ by adding
a suitable constant. With this normalization of $f$ the Perelman's invariant
is defined by \cite{P, CN}: 
\begin{equation*}
\mu _{0}:=-\log \left( \left( 4\pi \right) ^{-\frac{n}{2}}\int_{M}e^{-f}%
\right) <\infty .
\end{equation*}

In this section we prove that the volume of a complete noncompact, shrinking
gradient Ricci soliton is of at least linear growth. As pointed out in the
introduction, this result is optimal. The corresponding result for manifolds
with non-negative Ricci curvature was proved independently by Calabi \cite%
{Cl} and Yau \cite{Y2}.

\begin{theorem}
\label{Vol_Shrink}Let $\left( M,g,f\right) $ be a complete noncompact,
shrinking gradient Ricci soliton. Then there exists a constant $C>0$
depending only on dimension $n$ and the Perelman's invariant $\mu _{0}$
defined above such that 
\begin{equation*}
Vol\left( B_{p}\left( r\right) \right) \geq C\,r
\end{equation*}%
for all $r\geq r_0,$ where $p$ is a minimum point of $f$ and $r_0$ depends only on $n$.
\end{theorem}

Our proof of Theorem \ref{Vol_Shrink} involves Perelman's ideas in \cite{P,
ST} and a logarithmic Sobolev inequality for shrinking gradient Ricci
solitons \cite{CN}. The main difference here is that no extra assumptions
are imposed on the curvature.

The proof of the Theorem consists of several steps. We begin by proving that
the volume of unit balls decay at most exponentially on a noncompact
shrinking gradient Ricci soliton. This may be of independent interest. We
then show that the volume of any noncompact shrinking gradient Ricci soliton
must be infinite in Lemma \ref{Infinite}, a fact which also appeared in \cite%
{Ca1}. However, our approach here is different. With this fact, we then
complete our argument for Theorem \ref{Vol_Shrink}.

\begin{lemma}
\label{Decay} Let $\left( M,g,f\right) $ be a complete noncompact, shrinking
gradient Ricci soliton. Then there exists a constant $C(n)>0$ depending only
on dimension $n$ such that

\begin{equation*}
V\left( B_{x}\left( 1\right) \right) \geq V\left( B_{p}\left( 1\right)
\right)\,e^{-C(n)\,d\left( p,x\right) }
\end{equation*}
for all $x\in M,$ where $p$ is a minimum point of $f.$
\end{lemma}

\begin{proof}[Proof of Lemma \protect\ref{Decay}]
Take any point $x\in M$ and express the volume form in the geodesic polar
coordinates centered at $x$ as 
\begin{equation*}
dV|_{\exp _{x}\left( r\xi \right) }=J\left( x,r,\xi \right) drd\xi
\end{equation*}%
for $r>0$ and $\xi \in S_{x}M,$ a unit tangent vector at $x.$ In the
following, we will omit the dependence of these quantities on $\xi .$ Let $%
x\in M$ be arbitrary and 
\begin{equation*}
r:=d\left( p,x\right) .
\end{equation*}%
Notice that by the triangle inequality, $B_{p}\left( 1\right) \subset
B_{x}\left( r+1\right) \backslash B_{x}\left( r-1\right) .$ Let $\gamma
\left( s\right) $ be a minimizing geodesic starting from $x$, such that $%
\gamma \left( 0\right) =x$ and $\gamma \left( T\right) \in B_{p}\left(
1\right) ,$ for some%
\begin{equation}
r-1\leq T\leq r+1.  \label{i16'}
\end{equation}

Along $\gamma $ we have, by a standard formula:%
\begin{equation*}
\frac{J^{\prime }}{J}\left( t\right) \leq \frac{n-1}{t}-\frac{1}{t^{2}}%
\int_{0}^{t}s^{2}Ric\left( \gamma ^{\prime }\left( s\right) ,\gamma ^{\prime
}\left( s\right) \right) ds.
\end{equation*}%
Using the soliton equation that 
\begin{equation*}
Ric\left( \gamma ^{\prime }\left( s\right) ,\gamma ^{\prime }\left( s\right)
\right) =\frac{1}{2}-Hess\left( f\right) \left( \gamma ^{\prime }\left(
s\right) ,\gamma ^{\prime }\left( s\right) \right)
\end{equation*}%
and then integrating by parts we obtain: 
\begin{equation}
\frac{J^{\prime }}{J}\left( t\right) \leq \frac{n-1}{t}-\frac{1}{6}%
t+f^{\prime }\left( t\right) -\frac{2}{t^{2}}\int_{0}^{t}sf^{\prime }\left(
s\right) ds,  \label{i17}
\end{equation}%
where $f\left( s\right) :=f\left( \gamma \left( s\right) \right) $ and $%
0\leq t\leq T$. We note that by the triangle inequality 
\begin{equation*}
d\left( p,\gamma \left( s\right) \right) \leq d\left( p,\gamma \left(
T\right) \right) +d\left( \gamma \left( T\right) ,\gamma \left( s\right)
\right) \leq T-s+1
\end{equation*}%
and 
\begin{equation*}
d\left( p,\gamma \left( s\right) \right) \geq d\left( \gamma \left( T\right)
,\gamma \left( s\right) \right) -d\left( p,\gamma \left( T\right) \right)
\geq T-s-1.
\end{equation*}

So by the estimates of $f$ proved in \cite{CZ, HM} together with (\ref{i16'}%
) we obtain that 
\begin{eqnarray}
\frac{1}{4}\left[ \left( r-s-c\left( n\right) \right) _{+}\right] ^{2} &\leq
&f\left( \gamma \left( s\right) \right) \leq \frac{1}{4}\left( r-s+c\left(
n\right) \right) ^{2}\ \ \ \text{and}  \label{i18} \\
\left\vert \nabla f\right\vert \left( \gamma \left( s\right) \right) &\leq &%
\frac{1}{2}\left( r-s+c\left( n\right) \right) ,  \notag
\end{eqnarray}%
where $c\left( n\right) $ is a constant depending only on $n$ and $%
a_{+}:=\max \left\{ a,0\right\} .$ In the following, we will denote by $c$ a
constant depending only on $n,$ which may change from line to line.

Plugging this into (\ref{i17}) results in%
\begin{eqnarray*}
\frac{J^{\prime }}{J}\left( t\right) &\leq &\frac{n-1}{t}-\frac{1}{6}%
t+f^{\prime }\left( t\right) +\frac{1}{t^{2}}\int_{0}^{t}s\left(
r-s+c\right) ds \\
&=&\frac{n-1}{t}+c+f^{\prime }\left( t\right) +\frac{1}{2}\left( r-t\right) .
\end{eqnarray*}%
Integrating from $t=1$ to $t=T$ we obtain that 
\begin{equation}
\log \frac{J\left( T\right) }{J\left( 1\right) }\leq cT+f\left( T\right)
-f\left( 1\right) +\frac{1}{2}\left( rT-\frac{1}{2}T^{2}\right) .
\label{i18'}
\end{equation}%
Notice that $f\left( T\right) \leq c$ as $\gamma \left( T\right) \in
B_{p}\left( 1\right) .$ On the other hand, by (\ref{i18}), 
\begin{equation*}
f\left( 1\right) =f\left( \gamma \left( 1\right) \right) \geq \frac{1}{4}%
r^{2}-c\left( n\right) r.
\end{equation*}%
Hence, combining with (\ref{i16'}), one sees from (\ref{i18'}) that 
\begin{equation*}
\log \frac{J\left( T\right) }{J\left( 1\right) }\leq cr.
\end{equation*}%
In other words, 
\begin{equation*}
J\left( x,T,\xi \right) \leq e^{cd\left( p,x\right) }J\left( x,1,\xi \right)
,\ \ \text{whenever \ }\exp _{x}\left( T\xi \right) \in B_{p}\left( 1\right)
.
\end{equation*}%
By integrating this over a subset of $S_{x}M$ consisting of all unit tangent
vectors $\xi $ so that $\exp _{x}\left( T\xi \right) \in B_{p}\left(
1\right) $ for some $T,$ it follows that 
\begin{equation*}
V\left( B_{p}\left( 1\right) \right) \leq e^{cd\left( p,x\right) }A\left(
\partial B_{x}\left( 1\right) \right)
\end{equation*}%
for a constant $c$ depending only on $n.$

The same argument implies in fact that 
\begin{equation*}
V\left( B_{p}\left( 1\right) \right) \leq e^{cd\left( p,x\right) }A\left(
\partial B_{x}\left( \rho \right) \right)
\end{equation*}%
for all $1/2\leq \rho \leq 1.$ After integrating with respect to $\rho ,$ we
have 
\begin{equation*}
V\left( B_{x}\left( 1\right) \right) \geq V\left( B_{p}\left( 1\right)
\right) e^{-cd\left( p,x\right) }.
\end{equation*}%
The lemma is proved.
\end{proof}

We now establish the fact that the volume of a noncompact shrinking gradient
Ricci soliton is infinite.

\begin{lemma}
\label{Infinite} Let $\left( M,g,f\right) $ be a complete noncompact,
shrinking gradient Ricci soliton. Then $Vol\left( M\right) =\infty .$
\end{lemma}

\begin{proof}[Proof of Lemma \protect\ref{Infinite}]
Suppose to the contrary that 
\begin{equation*}
Vol\left( M\right) =V<\infty .
\end{equation*}

Define 
\begin{equation*}
\rho :=2\sqrt{f}\text{ \ \ and \ \ }D\left( t\right) :=\left\{ \rho \leq
t\right\} .
\end{equation*}%
According to \cite{CZ, HM}, 
\begin{equation*}
d\left( p,x\right) -a\leq \rho \left( x\right) \leq d\left( p,x\right) +a,
\end{equation*}%
for a constant $a>0$ depending only on $n.$ Moreover, it is easy to see that 
$\left\vert \nabla \rho \right\vert \leq 1$ on $M.$ We let 
\begin{equation*}
V\left( t\right) :=Vol\left( D\left( t\right) \right) \text{ \ \ and \ \ \ }%
\chi \left( t\right) :=\int_{D\left( t\right) }S.
\end{equation*}%
By the co-area formula, we have 
\begin{equation*}
V^{\prime }\left( t\right) =\int_{\partial D\left( t\right) }\frac{1}{%
\left\vert \nabla \rho \right\vert }\ \ \ \ \ \text{and \ \ }\chi ^{\prime
}\left( t\right) =\int_{\partial D\left( t\right) }\frac{S}{\left\vert
\nabla \rho \right\vert }.
\end{equation*}

Let us now recall the logarithmic Sobolev inequality from \cite{CN}, which
holds for any compactly supported Lipschitz function $u$ on $M.$ 
\begin{equation}
\int_{M}u^{2}\log u^{2}-\left( \int_{M}u^{2}\right) \log \left(
\int_{M}u^{2}\right) \leq \mu _{0}\int_{M}u^{2}+4\int_{M}\left\vert \nabla
u\right\vert ^{2}+\int_{M}Su^{2}.  \label{i1}
\end{equation}

For $t\geq 2,$ we define $u_{t}:\mathbb{R}\rightarrow \mathbb{R}$ by 
\begin{equation}
u_{t}\left( s\right) =\left\{ 
\begin{array}{c}
1 \\ 
s-\left( t-1\right) \\ 
0%
\end{array}%
\right. 
\begin{array}{l}
\text{on }s\geq t \\ 
\text{on }t-1\leq s\leq t \\ 
\text{on }0\leq s\leq t-1%
\end{array}
\label{i1'}
\end{equation}%
and a function on $M$ by $u_{t}\left( x\right) :=u_{t}\left( \rho \left(
x\right) \right) .$

Using the fact that $S\geq 0$ (see \cite{Ch, Ca1}) and $\chi (\infty
)=\int_{M}S\leq \frac{n}{2}Vol(M)<\infty $ $\ $(see \cite{CZ}), one easily
justifies that such $u_t$ is admissible in the preceding logarithmic Sobolev
inequality (\ref{i1}). Let us denote 
\begin{equation*}
y\left( t\right) :=\int_{M}u_{t}^2.
\end{equation*}%
The log-Sobolev inequality applied to $u_{t}$ implies%
\begin{equation*}
-y\left( t\right) \log y\left( t\right) \leq C\left( Vol\left( M\right)
-V\left( t-1\right) \right) +\int_{M}Su_{t}^{2},
\end{equation*}%
where $C$ depends only on $n$ and the Perelman's invariant $\mu _{0}.$ We
have also used above the elementary inequality $u_{t}^{2}\log u_{t}^{2}\geq -%
\frac{1}{e}.$

Since $y\left( t\right) \geq Vol\left( M\right) -V\left( t\right) ,$ we
obtain%
\begin{equation}
-y\left( t\right) \log y\left( t\right) \leq Cy\left( t-1\right)
+\int_{M}Su_{t}^{2}.  \label{i2}
\end{equation}%
We now wish to express the term $\int_{M}Su_{t}^{2}$ from (\ref{i2}) in
terms of $y\left( t\right) $. For any $T>t,$ since $S=\frac{n}{2}-\Delta f,$
we have%
\begin{eqnarray}
\int_{D\left( T\right) }Su_{t}^{2} &=&\frac{n}{2}\int_{D\left( T\right)
}u_{t}^{2}-\int_{D\left( T\right) }\left( \Delta f\right) \cdot u_{t}^{2}
\label{i3} \\
&=&\frac{n}{2}\int_{D\left( T\right) }u_{t}^{2}+\int_{D\left( T\right)
}\left\langle \nabla f,\nabla u_{t}^{2}\right\rangle -\int_{\partial D\left(
T\right) }\frac{\partial f}{\partial \nu }u_{t}^{2},  \notag
\end{eqnarray}%
where $\nu $ is the unit normal to $\partial D\left( t\right) .$ In fact,
since $\frac{\partial }{\partial \nu }=\frac{\nabla \rho }{\left\vert \nabla
\rho \right\vert },$ it follows that $\frac{\partial f}{\partial \nu }=\frac{%
1}{2}\rho \left\vert \nabla \rho \right\vert \geq 0.$ Moreover, observe that 
$\left\langle \nabla f,\nabla u_{t}^{2}\right\rangle $ has support in $%
D\left( t\right) \backslash D\left( t-1\right) $ and in that region we have 
\begin{eqnarray*}
\left\langle \nabla f,\nabla u_{t}^{2}\right\rangle &=&2\left\langle \nabla
f,\nabla \rho \right\rangle u_{t}=\rho \left\vert \nabla \rho \right\vert
^{2}u_{t} \\
&\leq &\rho u_{t}.
\end{eqnarray*}%
Using this in (\ref{i3}) we find that%
\begin{equation*}
\int_{D\left( T\right) }Su_{t}^{2}\leq \frac{n}{2}\int_{D\left( T\right)
}u_{t}^{2}+\int_{D\left( t\right) \backslash D\left( t-1\right) }\rho u_{t}.
\end{equation*}%
Hence, we can let $T\rightarrow \infty $ in the above and get%
\begin{equation}
\int_{M}Su_{t}^{2}\leq \frac{n}{2}y\left( t\right) +t\int_{D\left( t\right)
\backslash D\left( t-1\right) }u_{t}.  \label{i4}
\end{equation}%
By a direct calculation it follows that%
\begin{equation}
\frac{d}{dt}y\left( t\right) =\frac{d}{dt}\int_{M}u_{t}^{2}=-2\int_{D\left(
t\right) \backslash D\left( t-1\right) }u_{t}.  \label{i4'}
\end{equation}

Therefore, combining this with (\ref{i4}) we get 
\begin{equation*}
\int_{M}Su_{t}^{2}\leq \frac{n}{2}y\left( t\right) -\frac{1}{2}ty^{\prime
}\left( t\right) .
\end{equation*}

We use this in (\ref{i2}) to conclude that 
\begin{equation}
ty^{\prime }\left( t\right) -2y\left( t\right) \log y\left( t\right) \leq
Cy\left( t-1\right)  \label{i10}
\end{equation}%
for a constant $C$ depending only on $n$ and the Perelman's invariant $\mu
_{0}.$ This inequality is true for $t\geq c\left( n\right) ,$ where $c\left(
n\right) $ is a large enough constant so that $D\left( t\right) $ are
non-empty.

In the following, we will use this differential inequality to show that the
function $y(t)$ decays exponentially with arbitrarily large exponent. Here,
our argument is inspired by \cite{PSSW}. We first show $y(t)$ is of
exponential decay of some order. We let 
\begin{equation*}
\delta :=e^{-C}>0,
\end{equation*}%
where $C$ is the constant from (\ref{i10}). There exists $\varepsilon >0$
sufficiently small so that 
\begin{equation*}
y\left( \frac{1}{\varepsilon }\right) <\delta e^{-2}.
\end{equation*}%
Indeed, this is because $Vol\left( M\right) <\infty ,$ hence 
\begin{equation*}
\lim_{t\rightarrow \infty }y\left( t\right) =0.
\end{equation*}%
Clearly, we can assume $e^{\varepsilon }<2$. Let us define $t_{0}:=\frac{1}{%
\varepsilon }.$ We claim that 
\begin{equation}
y\left( t\right) <\delta e^{-\varepsilon t},\ \ \ \ \text{for any }t_{0}\leq
t\leq t_{0}+1.  \label{i11}
\end{equation}%
Indeed, by (\ref{i4'}) we know that $y\left( t\right) $ is decreasing in $t,$
therefore 
\begin{equation*}
y\left( t\right) \leq y\left( t_{0}\right) =y\left( \frac{1}{\varepsilon }%
\right) <\delta e^{-2},
\end{equation*}%
for any $t_{0}\leq t\leq t_{0}+1.$ However, since $t_{0}\varepsilon =1,$ we
can write 
\begin{equation*}
\delta e^{-2}=\delta e^{-2\varepsilon t_{0}}<\delta e^{-\varepsilon t},
\end{equation*}%
where the last inequality is true because $t\leq t_{0}+1 < 2t_{0}.$ Hence,
this proves that (\ref{i11}) is true.

Now we claim that 
\begin{equation}
y\left( t\right) <\delta e^{-\varepsilon t},  \label{i12}
\end{equation}%
for any $t_{0}\leq t.$ If (\ref{i12}) fails to be true for all $t\geq t_{0},$
there exists a first $t=r$ so that $y\left( r\right) =\delta e^{-\varepsilon
r}.$ Then the choice of $r$ implies 
\begin{eqnarray*}
y\left( r\right) &=&\delta e^{-\varepsilon r} \\
y^{\prime }\left( r\right) &\geq &-\varepsilon \delta e^{-\varepsilon r}.
\end{eqnarray*}

Since (\ref{i12}) is true for $t\leq t_{0}+1,$ we know that $r-1\geq t_{0}.$
Consequently, $y\left( r-1\right) \leq \delta e^{-\varepsilon \left(
r-1\right) }.$ Now (\ref{i10}) for $t=r$ implies that%
\begin{eqnarray*}
-\varepsilon \delta re^{-\varepsilon r}+2\delta e^{-\varepsilon r}\left(
-\log \delta +\varepsilon r\right) &\leq &ry^{\prime }\left( r\right)
-2y\left( r\right) \log y\left( r\right) \\
&\leq &Cy\left( t-1\right) \\
&\leq &C\delta e^{\varepsilon }e^{-\varepsilon r}.
\end{eqnarray*}%
Simplifying this gives%
\begin{equation*}
\varepsilon r-2\log \delta \leq Ce^{\varepsilon }.
\end{equation*}%
However, $\varepsilon r\geq \varepsilon t_{0}=1$ and $e^{\varepsilon }<2$
from the choice of $\varepsilon .$ Since by definition $\delta :=e^{-C},$
this is a contradiction.

Hence (\ref{i12}) is true for all $t\geq t_{0}.$ In particular, 
\begin{equation}
y\left( t\right) \leq e^{-\varepsilon t}\ \ \text{for all }t\geq t_{0}.
\label{i14}
\end{equation}

We have thus shown that $y(t)$ decays exponentially. We now show that $y(t)$
has arbitrarily large exponential decay rate.

Let us prove by induction on $m$ that there exists $t_{m}$ such that 
\begin{equation}
y\left( t\right) \leq e^{-\left( \frac{3}{2}\right) ^{m}\varepsilon t}\text{
\ for all }t\geq t_{m}.  \label{i14'}
\end{equation}

Clearly, (\ref{i14}) means that (\ref{i14'}) is true for $m=0.$ We now
assume that (\ref{i14'}) is true for $m\geq 0$ and prove so for $m+1.$

For a constant $t_{m+1}\geq t_{m}$ to be picked later, let $A$ be a constant
so that 
\begin{equation*}
y\left( t\right) <Ae^{-\frac{7}{4}\left( \frac{3}{2}\right) ^{m}\varepsilon
t},\ \ \text{for all }t_{m+1}\leq t\leq t_{m+1}+1.
\end{equation*}%
If this inequality holds true for all $t\geq t_{m+1},$ then by possibly
renaming $t_{m+1}$ to be a larger number, (\ref{i14'}) holds for $m+1$ and
the induction is complete. Otherwise, there exists the first $r>t_{m+1}+1$
so that $y\left( r\right) =Ae^{-\frac{7}{4}\left( \frac{3}{2}\right)
^{m}\varepsilon r}.$ Then, 
\begin{eqnarray}
y\left( r\right) &=&Ae^{-\frac{7}{4}\left( \frac{3}{2}\right)
^{m}\varepsilon r}  \label{i14''} \\
y^{\prime }\left( r\right) &\geq &-\frac{7}{4}\left( \frac{3}{2}\right)
^{m}A\varepsilon e^{-\frac{7}{4}\left( \frac{3}{2}\right) ^{m}\varepsilon r}.
\notag
\end{eqnarray}%
Therefore, by (\ref{i10}) we get 
\begin{equation}
-\frac{7}{4}\left( \frac{3}{2}\right) ^{m}A\varepsilon e^{-\frac{7}{4}\left( 
\frac{3}{2}\right) ^{m}\varepsilon r}\leq 2\frac{1}{r}y\left( r\right) \log
y\left( r\right) +\frac{C}{r}y\left( r-1\right) .  \label{i14'''}
\end{equation}%
Using the induction hypothesis, we know 
\begin{equation*}
2\log y\left( r\right) \leq -2\left( \frac{3}{2}\right) ^{m}\varepsilon r.
\end{equation*}%
From (\ref{i14''}) and (\ref{i14'''}), it follows that 
\begin{equation*}
-\frac{7}{4}\left( \frac{3}{2}\right) ^{m}A\varepsilon e^{-\frac{7}{4}\left( 
\frac{3}{2}\right) ^{m}\varepsilon r}\leq -\left( 2\left( \frac{3}{2}\right)
^{m}\varepsilon -\frac{C}{r}e^{\frac{7}{4}\left( \frac{3}{2}\right)
^{m}\varepsilon }\right) Ae^{-\frac{7}{4}\left( \frac{3}{2}\right)
^{m}\varepsilon r}.
\end{equation*}%
After simplifying this inequality, we get 
\begin{equation*}
\frac{1}{4}\left( \frac{3}{2}\right) ^{m}\varepsilon \leq \frac{C}{r}e^{%
\frac{7}{4}\left( \frac{3}{2}\right) ^{m}\varepsilon }\leq \frac{C}{t_{m+1}+1%
}e^{\frac{7}{4}\left( \frac{3}{2}\right) ^{m}\varepsilon }.
\end{equation*}%
This is not possible if we pick $t_{m+1}$ large enough. Therefore,

\begin{equation*}
y\left( t\right) <Ae^{-\frac{7}{4}\left( \frac{3}{2}\right) ^{m}\varepsilon
t},\ \ \text{for all }t\geq t_{m+1}.
\end{equation*}%
By choosing an even larger $t_{m+1}$ if necessary, we have 
\begin{equation*}
y\left( t\right) <e^{-\left( \frac{3}{2}\right) ^{m+1}\varepsilon t},\ \ 
\text{for all }t\geq t_{m+1}.
\end{equation*}%
In other words, (\ref{i14'}) is true for any $m.$ Consequently, for any $%
a>0, $ 
\begin{equation}
y\left( t\right) \leq C\left( a\right) e^{-at}\ \ \ \text{for all }t.
\label{i16}
\end{equation}

However, by Lemma \ref{Decay}, one sees that $y(t)$ decays at most
exponentially with a fixed rate. Indeed, we may choose $c(n)>0$ such that\
the geodesic ball 
\begin{equation*}
B_{x}(1)\subset D(t+c(n)+1)\setminus D(t-c(n)-1),
\end{equation*}%
where $d\left( p,x\right) =t$. Then, by Lemma \ref{Decay}, 
\begin{equation*}
y(t-c(n)-2)\geq V(B_{x}(1))\geq C_{1}(n)\,V(B_{p}(1))\,e^{-C_{2}(n)t},
\end{equation*}%
and this contradicts (\ref{i16}). This contradiction necessarily implies $%
Vol(M)=\infty .$ The proof is completed.
\end{proof}

We are now ready to prove Theorem \ref{Vol_Shrink}.

\begin{proof}[Proof of Theorem \protect\ref{Vol_Shrink}]
Recall $D\left( r\right) :=\left\{ \rho \leq r\right\} ,$ where $\rho :=2%
\sqrt{f}.$ According to \cite{CZ}, 
\begin{equation*}
d\left( p,x\right) -a\leq \rho \left( x\right) \leq d\left( p,x\right) +a.
\end{equation*}
with constant $a$ depending only on the dimension $n.$ So we may choose $%
r_{0}>100$ only depending on $n$ such that $D(r)$ has positive measure for $%
r\geq r_{0}.$

Let $V\left( t\right) :=Vol\left( D\left( t\right) \right) $ and $\chi
\left( t\right) :=\int_{D\left( t\right) }S.$ Then by \cite{CZ},

\begin{equation*}
\chi \left( t\right) \leq \frac{n}{2}V\left( t\right)
\end{equation*}%
for any $t>0.$ To prove Theorem \ref{Vol_Shrink}, it is sufficient to show
that $V\left(r\right) \geq C\,r$ for all $r\geq 2\,r_{0}$ for some positive
constant $C$ depending only on $n$ and $\mu _{0}.$

We are going to prove this by contradiction. Let us assume that for $%
\varepsilon >0$ small depending only on $n$ and $\mu _{0}$ there exists $%
r\geq 2r_{0}$ such that 
\begin{equation}
V\left( r\right) \leq \varepsilon r.  \label{vs0}
\end{equation}%
The choice of $\varepsilon $ will be made clear in the proof. Without loss
of generality, we may assume that $r\in \mathbb{N}$. Consider the following
set of positive integers:%
\begin{equation*}
\Omega :=\left\{ k\in \mathbb{N}\text{ }:\ V\left( t\right) \leq
2\varepsilon t\ \ \text{for all integers }r\leq t\leq k\right\} .
\end{equation*}

Clearly, $r\in \Omega .$ We now prove that in fact any integer $k\geq r$ is
in $\Omega ,$ which follows from the following claim.

\begin{claim}
\label{clm1} $k+1\in \Omega $ whenever $k\in \Omega .$
\end{claim}

Most of our argument is devoted to proving Claim \ref{clm1}. The following
logarithmic Sobolev inequality established by Carillo and Ni \cite{CN} will
again be central to our proof. 
\begin{equation}
\int_{M}u^{2}\log u^{2}-\left( \int_{M}u^{2}\right) \left( \log
\int_{M}u^{2}\right) \leq \mu
_{0}\int_{M}u^{2}+\int_{M}Su^{2}+4\int_{M}\left\vert \nabla u\right\vert ^{2}
\label{vs}
\end{equation}%
for any $u\in C_{0}^{\infty }\left( M\right) ,$ where $\mu _{0}$ is the
Perelman's invariant. Note that the scalar curvature $S\geq 0$ by \cite{Ch,
Ca1}.

For $t\geq 2r_{0},$ define function $u$ by 
\begin{equation*}
u\left( x\right) =\left\{ 
\begin{array}{c}
1 \\ 
t+2-\rho \left( x\right) \\ 
\rho \left( x\right) -\left( t-1\right) \\ 
0%
\end{array}%
\right. 
\begin{array}{l}
\text{on }D\left( t+1\right) \backslash D\left( t\right) \\ 
\text{on }D\left( t+2\right) \backslash D\left( t+1\right) \\ 
\text{on }D\left( t\right) \backslash D\left( t-1\right) \\ 
\text{otherwise}%
\end{array}%
\end{equation*}%
Obviously, $u$ is Lipschitz with compact support. Plugging $u$ into the
preceding logarithmic Sobolev inequality and noting that $x\log x\geq -\frac{%
1}{e}$ for any $x>0,$ we conclude

\begin{gather}
-\left( \int_{M}u^{2}\right) \log \left( V\left( t+2\right) -V\left(
t-1\right) \right) \leq C_{0}\left( V\left( t+2\right) -V\left( t-1\right)
\right)  \label{vs1} \\
+\left( \chi \left( t+2\right) -\chi \left( t-1\right) \right) ,  \notag
\end{gather}%
where $C_{0}:=\mu _{0}+4+\frac{1}{e}.$

On the other hand, according to \cite{CZ},%
\begin{equation}
\frac{V\left( t+1\right) }{\left( t+1\right) ^{n}}-\frac{V\left( t\right) }{%
t^{n}}\leq 4\frac{\chi \left( t+1\right) }{\left( t+1\right) ^{n+2}},\text{
\ for any }t>\sqrt{2\left( n+2\right) }.  \label{vs2}
\end{equation}%
Since $\chi \left( t\right) \leq \frac{n}{2}V\left( t\right) ,$ we get from (%
\ref{vs2}) that 
\begin{eqnarray*}
V\left( t+1\right) -V\left( t\right) &\leq &\left( t+1\right) ^{n}\left( 
\frac{V\left( t\right) }{t^{n}}+2n\frac{V\left( t+1\right) }{\left(
t+1\right) ^{n+2}}\right) -V\left( t\right) \\
&=&\frac{\left( t+1\right) ^{n}-t^{n}}{t^{n}}V\left( t\right) +2n\frac{%
V\left( t+1\right) }{\left( t+1\right) ^{2}}.
\end{eqnarray*}%
Observe that 
\begin{equation*}
\frac{\left( t+1\right) ^{n}-t^{n}}{t^{n}}\leq \frac{2^{n}}{t}.
\end{equation*}%
So we have 
\begin{equation}
V\left( t+1\right) -V\left( t\right) \leq 2^{n}\frac{V\left( t\right) }{t}+2n%
\frac{V\left( t+1\right) }{\left( t+1\right) ^{2}}.  \label{vs3}
\end{equation}%
In particular, there exists $C\left( n\right) $ depending only on $n$ such
that 
\begin{equation}
V\left( t+1\right) \leq 2V\left( t\right)  \label{vs4}
\end{equation}%
for all $t\geq C\left( n\right) .$

Plugging this back into (\ref{vs3}), we find that 
\begin{equation}
V\left( t+1\right) -V\left( t\right) \leq C_{1}\frac{V\left( t\right) }{t}
\label{vs5}
\end{equation}%
for all $t\geq C\left( n\right) ,$ where $C_{1}$ depends only on $n.$
Moreover, using (\ref{vs4}) and (\ref{vs5}) or, alternatively using the same
argument as above, we also obtain 
\begin{equation*}
V\left( t+2\right) -V\left( t-1\right) \leq C_{2}\frac{V\left( t\right) }{t}
\end{equation*}%
for all $t\geq C\left( n\right) ,$ where $C_{2}$ depends only on $n.$

Now for all integers $r\leq t\leq k,$ note that $t\in \Omega .$ So $\frac{%
V\left( t\right) }{t}\leq 2\,\varepsilon ,$ which leads to 
\begin{eqnarray*}
V\left( t+1\right) -V\left( t\right) &\leq &2C_{1}\varepsilon , \\
V\left( t+2\right) -V\left( t-1\right) &\leq &2C_{2}\varepsilon .
\end{eqnarray*}%
Plugging this into (\ref{vs1}), we arrive at 
\begin{gather}
\left( V\left( t+1\right) -V\left( t\right) \right) \log \left(
2C_{2}\varepsilon \right) ^{-1}\leq C_{0}\left( V\left( t+2\right) -V\left(
t-1\right) \right)  \label{vs5'} \\
+\left( \chi \left( t+2\right) -\chi \left( t-1\right) \right)  \notag
\end{gather}%
provided $\varepsilon $ is chosen to satisfy $2C_{2}\varepsilon <1.$

Iterating (\ref{vs5'}) from $t=r$ to $t=k$ and summing up all the resulting
inequalities, we get 
\begin{equation*}
\left( V\left( k+1\right) -V\left( r\right) \right) \log \left(
2C_{2}\varepsilon \right) ^{-1}\leq 3C_{0}\left( V\left( k+2\right) +\chi
\left( k+2\right) \right) .
\end{equation*}%
Using again that $\chi \left( k+2\right) \leq \frac{n}{2}V\left( k+2\right) $
and also (\ref{vs4}), we obtain that 
\begin{equation*}
\left( V\left( k+1\right) -V\left( r\right) \right) \log \left(
2C_{2}\varepsilon \right) ^{-1}\leq C_{3}V\left( k+1\right) ,
\end{equation*}%
where $C_{3}$ depends only on dimension $n$ and the Perelman's invariant $%
\mu _{0}.$ Rearranging the terms, and using (\ref{vs0}), we get%
\begin{eqnarray}
V\left( k+1\right) &\leq &V\left( r\right) \frac{\log \left(
2C_{2}\varepsilon \right) ^{-1}}{\log \left( 2C_{2}\varepsilon \right)
^{-1}-C_{3}}  \label{v5''} \\
&\leq &\varepsilon \,r\frac{\log \left( 2C_{2}\varepsilon \right) ^{-1}}{%
\log \left( 2C_{2}\varepsilon \right) ^{-1}-C_{3}}.  \notag
\end{eqnarray}

Let us choose $\varepsilon $ small enough, depending on $n$ and $\mu _{0}$,
so that 
\begin{equation*}
\frac{\log \left( 2C_{2}\varepsilon \right) ^{-1}}{\log \left(
2C_{2}\varepsilon \right) ^{-1}-C_{3}}\leq 2.
\end{equation*}
From (\ref{v5''}) we conclude that 
\begin{equation}
V\left( k+1\right) \leq 2\varepsilon r,\ \ \ \ \text{for any }k\in \Omega .
\label{vs6}
\end{equation}%
Since $r\leq \left( k+1\right) ,$ by (\ref{vs6}) this proves Claim \ref{clm1}%
.

We have thus proved that%
\begin{equation*}
\Omega =\left\{ k\in \mathbb{N\ }:\text{ \ }k\geq r\right\} .
\end{equation*}
However, (\ref{vs6}) now implies that $V\left( k\right) \leq 2\varepsilon r,$
for any integer $k\geq r.$ This implies that the volume of $M$ is finite,
which is a contradiction to Lemma \ref{Infinite}.

This contradiction indicates there exists no such $r>r_{0}$ such that $%
V\left( r\right) \leq \varepsilon \,r$ with the $\varepsilon >0$ chosen in
the preceding argument, which depends only on $n$ and $\mu _{0}.$ That is, $%
V\left( r\right) \geq \varepsilon \,r$ \ for $r>r_{0}.$ Theorem \ref%
{Vol_Shrink} is proved.
\end{proof}

{\small DEPARTMENT OF MATHEMATICS, COLUMBIA UNIVERSITY }

{\small NEW YORK, NY 10027}\newline
{\small E-mail address: omuntean@math.columbia.edu}

\bigskip

{\small SCHOOL OF MATHEMATICS, UNIVERSITY OF MINNESOTA }

{\small MINNEAPOLIS, MN 55455} \newline
{\small E-mail address: jiaping@math.umn.edu}


\begin{thebibliography}{99}
\bibitem{BE} D. Bakry and M. \'{E}mery, Diffusions hypercontractives. In
Seminaire de probabilites, XIX, 1983/84, volume 1123 of Lecture Notes in
Math., pages 177--206. Springer, Berlin, 1985.

\bibitem{Bu} P. Buser, A note on the isoperimetric constant, Ann. Sci. Ecole
Norm. Sup, 15 (1982), 213-230.

\bibitem{C} S.Y. Cheng, Eigenvalue comparison theorems and its geometric
applications, Math. Z. 143 (1975) 289-297.

\bibitem{Cl} E. Calabi, On manifolds with non-negative Ricci curvature II,
Notices Amer. Math. Soc., 22 (1975), A205.

\bibitem{Ca1} H.D. Cao, Geometry of complete gradient shrinking Ricci
solitons; arXiv:0903.3927

\bibitem{CZ} H.D. Cao and D. Zhou, On complete gradient shrinking Ricci
solitons, J. Differential Geom. 85 (2010), no. 2, 175-186.

\bibitem{CN} J. Carillo and L. Ni, Sharp logarithmic sobolev inequalities on
gradient solitons and applications. Comm. Anal. Geom. 17, 721--753 (2009)

\bibitem{Ch} B.L.\ Chen, Strong uniqueness of the Ricci flow, J.
Differential Geom. 82 (2009), no. 2, 362-382.

\bibitem{CLN} B. Chow, P. Lu and L. Ni, Hamilton's Ricci flow, Graduate
studies in mathematics, 2006.

\bibitem{HK} P. Haj\l asz and P. Koskela, Sobolev meets Poincar\'{e}, C. R.
Acad. Sci. Paris Sr. I Math. 320 (1995), 1211-1215.

\bibitem{H} R. Hamilton, The formation of singularities in the Ricci flow,
Surveys in Differential Geom. 2 (1995), 7-136, International Press.

\bibitem{HM} R. Haslhofer and R. Muller, A compactness theorem for complete
Ricci shrinkers, to appear in Geom. Funct. Anal.

\bibitem{L} P. Li, Harmonic functions and applications to complete
manifolds, lecture notes on personal web page.

\bibitem{Li1} P. Li, Lecture Notes on Geometric Analysis, Lecture Notes
Series No. 6, Research Institute of Mathematics, Global Analysis Research
Center, Seoul National University, Korea (1993).

\bibitem{LT} P. Li and L-F. Tam, Harmonic functions and the structure of
complete manifolds, J. Differential Geom. 35 (1992), 359-383.

\bibitem{LW0} P. Li and J. Wang, Complete manifolds with positive spectrum,
J. Differential Geom. 58 (2001) 501--534,

\bibitem{LW} P. Li and J. Wang, Complete manifolds with positive spectrum,
II. J. Differential. Geom. 62 (2002), 143-162.

\bibitem{LW1} P. Li and J. Wang, Connectedness at infinity of complete K\"{a}%
hler manifolds, Amer. J. Math. 131 (2009), 771-817.

\bibitem{LW2} P. Li and J. Wang, Weighted Poincar\'{e} inequality and
rigidity of complete manifolds. Ann. Scient. \'{E}c. Norm. Sup., 4e s\'{e}%
rie, t. 39 (2006), 921-982.

\bibitem{MW} O. Munteanu and J. Wang, Smooth metric measure spaces with
nonnegative curvature, Comm. Anal. Geom. 19 (2011), no. 3, 451-486.

\bibitem{P} G.\ Perelman, The entropy formula for the Ricci flow and its
geometric applications, arXiv:math. DG/0211159.

\bibitem{PW} P. Petersen and W. Wylie, Rigidity of gradient Ricci solitons.
Pacific Journal of Math., 241(2), pp. 329-345, 2009

\bibitem{PSSW} D.H. Phong, J. Song, J. Sturm, B.\ Weinkove, The Kahler-Ricci
flow and the $\bar{\partial}-$operator on vector fields, J. Differential
Geom. 81, 3 (2009), 631--647.

\bibitem{PRS} S. Pigoli, M. Rimoldi and A. Setti, Remarks on non-compact
gradient Ricci solitons, arXiv 0905.2868

\bibitem{SC1} L. Saloff-Coste, Aspects of Sobolev-Type Inequalities, London
Mathematical Society Lecture Notes Series, 2001

\bibitem{ST} N. Sesum and G. Tian, Bounding scalar curvature and diameter
along the Kahler Ricci flow (after Perelman), J. Inst. Math. Jussieu (2008),
no 3, 575-587.

\bibitem{SZ} Y. Su and H. Zhang, Rigidity of manifolds with Bakry-\'{E}mery
Ricci curvature bounded below, preprint, 2011

\bibitem{W} G. Wei and W. Wylie, Comparison geometry for the Bakry-\'{E}mery
Ricci tensor, J. Differential Geom (2009), 377-405.

\bibitem{Wu} J.-Y. Wu, Upper bounds on the first eigenvalue for a diffusion
operator via Bakry-\'{E}mery Ricci curvature II, arXiv:1010.4175.

\bibitem{Y} S.T. Yau, Harmonic functions on complete Riemannian manifolds,
Comm. Pure Appl. Math., 28 (1975), 201-228.

\bibitem{Y2} S.T. Yau, Some function-theoretic properties of complete
Riemannian manifolds and their applications to geometry, Indiana Univ. Math.
J., 25 (1976), 659-670.

\bibitem{Z} S. Zhang, On a Sharp Volume Estimate for Gradient Ricci Solitons
with Scalar Curvature Bounded Below, Acta Mathematica Sinica, May, 2011,
Vol. 27, No. 5, pp. 871--882
\end{thebibliography}
\end{document}